\documentclass[11pt]{amsart}
\usepackage{amsmath,mathtools}
\usepackage{graphicx, color, enumerate}
\usepackage{amscd}
\usepackage{amsmath}
\usepackage{amsfonts}
\usepackage{amssymb}
\usepackage{mathrsfs}
\usepackage{latexsym}

\usepackage{nameref,hyperref,url}
\usepackage[foot]{amsaddr}

\usepackage{pgfplots}
\pgfplotsset{compat=1.15}
\usetikzlibrary{arrows}
\definecolor{blue}{rgb}{0,0,1}
\definecolor{red}{rgb}{0.8,0,0}
\definecolor{green}{rgb}{0,0.4,0}
\newcommand\e\varepsilon

\newcommand\R{\mathbb R}

\newcommand\de\partial
\newcommand\weakto\rightharpoonup

\renewcommand\le\leqslant
\renewcommand\ge\geqslant
\renewcommand\a\alpha

\renewcommand\b\beta

\renewcommand\d\delta
\newcommand\vfi\varphi
\newcommand\g\gamma
\newcommand\gb\gamma
\renewcommand\l\lambda
\newcommand\n\nabla
\newcommand\s\sigma
\renewcommand\t\theta
\renewcommand\O\S

\newcommand\G\Gamma

\renewcommand\S\Sigma
\renewcommand\L\Lambda

\newcommand\N{\mathbb N}

\renewcommand\o\S

\def\bbm[#1]{\text{\boldmath $#1$}}

\renewcommand\leq{\leqslant}
\renewcommand\geq{\geslant}

\renewcommand\ss[1]{\mathfrak{#1}}
\newtheorem{theorem}{Theorem}[section]
\newtheorem{lemma}[theorem]{Lemma}

\newtheorem{definition}[theorem]{Definition}

\def\sideremark#1{\ifvmode\leavevmode\fi\vadjust{\vbox to0pt{\vss
\hbox to0pt{\hskip\hsize\hskip1em
\vbox{\hsize3cm\tiny\raggedright\pretolerance10000
\noindent #1\hfill}\hss}\vbox to8pt{\vfil}\vss}}}

\definecolor{edu}{rgb}{0,1,0.2}

\numberwithin{equation}{section}
\newcommand\eps{\varepsilon}
\renewcommand\geq{\geqslant}

\title[A system with general Hardy--Sobolev singular criticalities]
{Existence of solutions for a system with general Hardy--Sobolev singular criticalities}


\author{\'Angel Arroyo, Rafael L\'opez-Soriano, Alejandro Ortega}

\email[\'Angel Arroyo]{angelrene.arroyo@ua.es}
\email[Rafael López-Soriano]{ralopezs@ugr.es}%
\email[Alejandro Ortega ]{alejandro.ortega@mat.uned.es}

\address[\'A. Arroyo]{Departamento de Matem\'aticas, Universidad de Alicante, 03690 San Vicente del Raspeig, Alicante, Spain}
\address[R. López-Soriano]{Universidad de Granada, IMAG, Departamento de Análisis Matem\'atico,  Campus Fuentenueva, 18071 Granada, Spain}
\address[A. Ortega]{Dpto. de Matem\'aticas Fundamentales, Facultad de Ciencias, UNED, 28040 Madrid, Spain}

\begin{document}

\begin{abstract}
In this paper we study a class of Hardy--Sobolev type systems defined in $\mathbb{R}^N$ and coupled by a singular critical Hardy--Sobolev term. The main novelty of this work is that the orders of the singularities are independent and contained in a wide range.
By means of variational techniques, we will prove the existence of positive bound and ground states for such a system. In particular,  we find solutions as minimizers or Mountain--Pass critical points of the energy functional on the underlying Nehari manifold. 

\end{abstract}

\maketitle

\section{Introduction}\setcounter{equation}0

Let us consider a family of Hardy--Sobolev type elliptic systems containing critical Hardy--Sobolev terms with singularities of different orders, namely
\begin{equation}\label{system:alphabeta}
\left\{\begin{array}{ll}
\displaystyle -\Delta u - \lambda_1 \frac{u}{|x|^2}-\frac{u^{2^*_{s_1}-1}}{|x|^{s_1}}= \nu \alpha h(x) \frac{u^{\alpha-1}v^\beta}{|x|^{s_3}}   &\text{in }\mathbb{R}^N,\vspace{.3cm}\\
\displaystyle -\Delta v - \lambda_2 \frac{v}{|x|^2}-\frac{v^{2^*_{s_2}-1}}{|x|^{s_2}}= \nu \beta h(x) \frac{u^\alpha v^{\beta-1}}{|x|^{s_3}} &\text{in }\mathbb{R}^N,\vspace{.3cm}\\
u,v> 0 & \text{in }\mathbb{R}^N\setminus\{0\},
\end{array}\right.
\end{equation}
where $\lambda_1,\lambda_2\in(0,\Lambda_N)$, being $\Lambda_N=\dfrac{(N-2)^2}{4}$ the best constant in the Hardy inequality, $\nu>0$ and  $\alpha,\beta$ are real parameters such that 
\begin{equation}\label{alphabeta}
\alpha, \beta> 1 \qquad \mbox{ and } \qquad  \frac{\alpha}{2_{s_1}^*}+\frac{\beta}{2_{s_2}^*}\leq 1,
\end{equation}
and $s_1,s_2,s_3\in(0,2)$ satisfying 
\begin{equation}\label{eses}
s_3\ge s_1\frac{\alpha}{2_{s_1}^*}+s_2\frac{\beta}{2_{s_2}^*}.
\end{equation}
The value $2^*_{s}=\frac{2(N-s)}{N-2}$ denotes the critical Hardy--Sobolev exponent, whereas $h\gneq 0$ is a function such that
\begin{equation}\label{H1}
h\in L^{\mathfrak{p},\sigma}(\mathbb{R}^N),
\end{equation}
for appropriate $\mathfrak{p}>1$ and $\sigma\in(0,2)$, which will be defined in \eqref{sigmadef} below. Here $L^{p,s}(\mathbb{R}^N)$ denotes the weighted $L^p$-space of measurable functions $u$ such that
\[\| u\|_{p,s}:= \left(\int_{\mathbb{R}^N}\frac{|u|^p}{|x|^s}dx\right)^{\frac{1}{p}}<\infty.\]

Systems like \eqref{system:alphabeta} have been extensively studied in the last years. The case $s_1=s_2=s_3=0$ was addressed in \cite{AbFePe}, where it is proved the existence of positive bound and ground states under the assumption $h\in L^1(\mathbb{R}^N)\cap L^\infty(\mathbb{R}^N)$ if $\alpha+\beta<2_0^*$ (which in turn implies that $h\in L^q(\mathbb{R}^N)$ for every $q>1$) and $h\in L^{\infty}(\mathbb{R}^N)$ if $\alpha+\beta=2_0^*$. The same non-singular system, for a wider range of the parameters $\alpha$ and $\beta$, namely $1<\alpha<2\leq \beta$ and $\lambda_2<\lambda_1$ or $1<\beta<2\leq \alpha$ and $\lambda_1<\lambda_2$, it is considered in \cite{CoLSOr}, where the existence of positive bound and ground states is derived under the same assumptions on the function $h$. The special case $\alpha=2$ and $\beta=1$, arising from the so-called Schr\"odinger--Korteweg--de Vries model, is analyzed in \cite{CoLSOr2},where analogous existence results for positive bound and ground states are obtained. In \cite{LSO} the presence of the Hardy--Sobolev terms with $s_1=s_2=s_3=s\in(0,2)$ was recently studied, where the authors refined the assumptions on the function $h$. More precisely, they considered 
\begin{equation*}
h(x)\in L^{\mathfrak{p},\sigma}(\mathbb{R}^N)\qquad\text{with}\qquad \mathfrak{p}=\frac{2_s^*}{2_s^*-\alpha-\beta}\quad\text{and}\quad \sigma=s.
\end{equation*}
The main goal of this work is then to extend these existence results for bound and ground state solutions for the wider range $s_1, s_2, s_3 \in (0,2)$. Let us emphasized that, under the assumption $s_i\neq0$, the concentration phenomena could take place only at $0$ and $\infty$ as, since far from these points all the terms are subcritical in the Sobolev sense.

Existence results for nonlinear elliptic systems arising from some particular configurations of system \eqref{system:alphabeta} have attracted great attention in the last years due to their connection with some models of Physics such as Nonlinear Optics models, Quantum Mechanics models or Bose--Einstein condensates. For instance, the case $s_1=s_2=s_3=0$, is closely related to an important family of nonlinear elliptic systems, namely Schrödinger type systems of the form
\begin{equation*}
\left\{\begin{array}{ll}
\displaystyle{-\Delta u + V_1(x) u= \mu_1 u^{2p-1} + \nu u^{p-1} v^{p}}  &\text{in }\mathbb{R}^N,\vspace{.3cm}\\
-\Delta v + V_2(x) v = \mu_2 v^{2p-1} + \nu u^{p}v^{p-1} &\text{in }\mathbb{R}^N,\\
u,v\geq0& \text{in }\mathbb{R}^N,
\end{array}\right.
\end{equation*}
where $1<2p\le2_0^*$ with $N\ge 3$ and the function $V(x)$ represents the role of the potential. There is a huge amount of literature concerning existence and multiplicity of solutions to these systems for both subcritical and critical settings so it is complicated to give a complete list of references. We refer to \cite{AC2,ACR, BW, LinWei, LW, MMP, POMP, SIR}. Let us note that, if one considers constant potentials and the critical exponent $2p=2_0^*$, then the above system has only the trivial solution $(u,v)=(0,0)$ due to a Pohozaev type identity. The competitive case, namely $\nu<0$, has also been extensively studied from a variational point of view in \cite{ClappPistoia, ClappSzu, TaYou}.

Elliptic systems involving Hardy type potentials $V_j(x)=-\frac{\lambda_j}{|x|^2}$ arise in nonrelativistic Quantum Mechanics, molecular physics or combustion models. In this case the non-trivial solutions could be singular around the point $x=0$. We refer to \cite{BVG, FigPerRos, Kang2012, Kang2014, ZhongZou} among others. By means of an approach based on the moving planes method, some qualitative properties of the solutions to elliptic systems involving Hardy type potentials have been recently shown in \cite{EsLSSci} for the case $h(x)=1$ and $s_i=0$, $i=1,2,3$. In the special case $\lambda_1=\lambda_2$, the authors also obtained a classification criterion of the solutions in terms of the solutions of the uncoupled equations (see \eqref{decoupled} below).

In what concerns elliptic systems involving Hardy--Sobolev critical terms the related literature is much lower. Up to our knowledge, the first work dealing with Hardy--Sobolev critical singular systems in $\mathbb{R}^N$ is \cite{LSO}, where the authors analyze the system \eqref{system:alphabeta} by assuming all the singularities to be of the same order. Elliptic systems involving Hardy--Sobolev critical terms in bounded domains were analyzed in \cite{Bouchekif, Nyamoradi, Zhang} among others.

We analyze the existence of bound and ground states through a variational approach, namely, we look for critical points of the energy functional associated to system \eqref{system:alphabeta} (see \eqref{functalphabeta}). One of the key points of this variational approach is the role of the \textit{semitrivial} solutions as critical points of the energy functional. Actually, for any $\nu\in\mathbb{R}$, note that the system \eqref{system:alphabeta} has two solutions $(z_1,0)$ and $(0,z_2)$, where $z_i$ satisfies the equation  
\begin{equation}\label{decoupled}
-\Delta z_i - \lambda_i \frac{z_i}{|x|^2}=\frac{z^{2^*_{s_i}-1}}{|x|^{s_i}} \qquad\mbox{ and }   \qquad z_i>0\ \mbox{ in } \mathbb{R}^N \setminus \{0\}.
\end{equation}
The solutions $z_i$ were completely characterized in \cite{KangPeng}. Actually, the family of rescaled functions $z_{\mu}^{\lambda_i}$  (see \eqref{zeta}) related to $z^{\lambda_i}$ are the extreme functions for the Hardy--Sobolev inequality  \eqref{H-S_lambda}, that is, the functions at which the corresponding Hardy--Sobolev constant \eqref{Slambda} is attained. This result, along with the Concentration-Compactness Principle by P.-L. Lions (cf. \cite{Lions1,Lions2}) together with the characterization of concentration profiles of the functions $z_{\mu}^{\lambda}$ (cf. \cite{Li}) are pivotal in the variational approach employed in this work. 

The concentration phenomenon is a typical issue in critical problems reflecting the lack of compactness of the corresponding Sobolev embedding at the critical exponent. Since we are considering $s_i\neq 0$, for the critical case given by the exponents configuration
\begin{equation*}
\frac{\alpha}{2_{s_1}^*}+\frac{\beta}{2_{s_2}^*}=1,
\end{equation*}
the concentration can only take place at $0$ and $\infty$. This is where the function $h$ comes into play. As we will see the main role of the function $h$ will be that of eliminate the possible concentration at the points $0$ or $\infty$ with the aid of its summability and its behavior around $0$ and $\infty$. Therefore, in order to rule out this concentration phenomenon it is needed to assume the following hypotheses,
\begin{equation}\label{H}
	\widetilde h(x)=	\frac{h(x)}{|x|^{\sigma}} \text{ is continuous at $0$ and $\infty$ and } \widetilde h(0)=0 \mbox{ and } \lim_{x\to+\infty}\widetilde h(x)=0.
\end{equation}
In former works dealing with nonsingular terms, namely $s_i=0$, the concentration can occur at various points $x_j\in\mathbb{R}^N$, so to rule out the concentration requires additional assumptions, either on the magnitude of the coupling parameter or by assuming the function $h$ to be radial (cf. \cite{AbFePe, CoLSOr, CoLSOr2}). This last assumption not only eliminates the possibility of concentration at points $x_j\in\mathbb{R}^N$ but also ensures the ground states to be radial. We will come back to this assumption when proving existence of positive ground state solutions which are radially symmetric.

The role of the function $h$ is highlighted by considering different orders $s_1,s_2$ and $s_3$. In particular, the behavior around $0$ and $\infty$ (see Subsection \ref{roleh}) required to $h$ in order to cancel out the possible concentrations taking place there. On the other hand, considering this difference in the order of the singularities introduces some extra difficulties as the problem is no longer homogeneous.

\subsection{Main results}

Notice first that the problem \eqref{system:alphabeta} is the Euler--Lagrange system for the energy functional
\begin{equation}\label{functalphabeta}
\begin{split}
\mathcal{J}_\nu (u,v)=&\frac{1}{2} \int_{\R^N} \left( |\nabla u|^2 + |\nabla v|^2  \right) \, dx -\frac{\lambda_1}{2} \int_{\R^N} \dfrac{u^2}{|x|^2} \, dx   -\frac{\lambda_2}{2} \int_{\R^N} \dfrac{v^2}{|x|^2} \, dx  \vspace{0,7cm} \\
&- \frac{1}{2^*_{s_1}} \int_{\R^N} \frac{|u|^{2^*_{s_1}}}{|x|^{s_1}} \, dx - \frac{1}{2^*_{s_2}} \int_{\R^N} \frac{|v|^{2^*_{s_2}}}{|x|^{s_2}} \, dx -\nu \int_{\R^N} h(x) \frac{|u|^\alpha |v|^{\beta}}{|x|^{s_3}} \, dx  ,
\end{split}
\end{equation}
whose domain is the product space $\mathbb{D}=\mathcal{D}^{1,2} (\mathbb{R}^N)\times \mathcal{D}^{1,2} (\mathbb{R}^N)$ with $\mathcal{D}^{1,2} (\mathbb{R}^N)$ defined as the completion of $C_0^{\infty}(\R^N)$ under the norm
\begin{equation*}
\|u\|_{\mathcal{D}^{1,2}(\R^N)}^2=\int_{\R^N} \, |\nabla u|^2  \, dx  .
\end{equation*}

Along this paper we shall give some results concerning the existence of solutions by the study of the geometry of $\mathcal{J}_\nu$. Actually, we can provide some qualitative information about the minimizers depending on the behavior of the function $h$. More precisely, let us assume that
 \begin{equation}\label{Hradial}
 h \mbox{ is a radial function and non-increasing in } (0,\infty).
 \end{equation} 

In a first result we are able to relate the energy level to the size of the coupling factor. In particular, we will see that the minimum level decreases arbitrarily as $\nu$ increases. Therefore one can guarantee the existence of a minimizer.

\begin{theorem}\label{thm:nugrande}
Assume \eqref{H1} and either $\frac{\alpha}{2_{s_1}^*}+\frac{\beta}{2_{s_2}^*}< 1$ or $\frac{\alpha}{2_{s_1}^*}+\frac{\beta}{2_{s_2}^*}= 1$ satisfying \eqref{H}. Then there exists $\overline{\nu}>0$ such that the system \eqref{system:alphabeta} has a positive ground state $(\tilde{u},\tilde{v}) \in \mathbb{D}$ for every $\nu>\overline{\nu}$. Moreover, if \eqref{Hradial} holds, such a ground state is radially symmetric.
\end{theorem}

Next, we establish the existence of the ground states by studying the behavior of the \textit{semitrivial} solutions. Let $\mathfrak{C}(\lambda_i,s_i)$ be the energy levels of the solutions of the decoupled equations \eqref{decoupled} for $i=1,2$, (see \eqref{critical_levels} below). A first configuration is one that the minimum energy level between $(z_1,0)$ and $(0,z_2)$ is provided by a saddle point of $\mathcal{J}_{\nu}$, therefore the existence of a ground state of \eqref{system:alphabeta} follows.

\begin{theorem}\label{thm:lambdaground}
Assume \eqref{H1} and either $\frac{\alpha}{2_{s_1}^*}+\frac{\beta}{2_{s_2}^*}< 1$ or $\frac{\alpha}{2_{s_1}^*}+\frac{\beta}{2_{s_2}^*}= 1$ satisfying \eqref{H}. If one of the following statements is satisfied
\begin{itemize}
\item[$i)$] $\mathfrak{C}(\lambda_1,s_1) \le \mathfrak{C}(\lambda_2,s_2)$  and either $\beta=2$ and $\nu$ large enough or $\beta<2$,
\item[$ii)$] $\mathfrak{C}(\lambda_1,s_1) \ge \mathfrak{C}(\lambda_2,s_2)$ and either $\alpha=2$ and $\nu$ large enough or $\alpha<2$,
\end{itemize}
then system \eqref{system:alphabeta} admits a positive ground state $(\tilde{u},\tilde{v})\in\mathbb{D}$.

In particular, if $\max\{\alpha,\beta\}<2$ or $\max\{\alpha,\beta\}\le2$ with $\nu$ sufficiently large, then system \eqref{system:alphabeta} admits a positive ground state $(\tilde{u},\tilde{v})\in\mathbb{D}$.

Moreover, if \eqref{Hradial} holds, the ground state is radially symmetric.
\end{theorem}
Subsequently, one shall study the reverse setting. For instance, if $\mathfrak{C}(\lambda_1,s_1) > \mathfrak{C}(\lambda_2,s_2)$, the couple $(0,z_2)$ provides the minimum energy between the \textit{semitrivial} solutions. Actually, if either $\alpha>2$ or $\alpha=2$ with $\nu$ small enough, such a couple becomes a local minimum. Moreover, one can show the existence of a small enough $\nu>0$ such that the first {\it semi-trivial} solution is indeed a ground state.

Let us stress that the constant $\mathfrak{C}(\lambda,s)$ is decreasing in both $\lambda$ and $s$ (see \eqref{Slambda} and \eqref{critical_levels} below). In case of having $s_1=s_2$ then the order of the constants $\mathfrak{C}(\lambda_1,s)$ and $\mathfrak{C}(\lambda_2,s)$ is determined by that of $\lambda_1$ and $\lambda_2$, namely, if $\lambda_1\gtreqless \lambda_2$ then $\mathfrak{C}(\lambda_1,s)\lesseqgtr\mathfrak{C}(\lambda_2,s)$ (cf. \cite{LSO}). Nevertheless, in the general case $s_1\neq s_2$ we have a wider range of configurations and, for example, we could have $\mathfrak{C}(\lambda_1,s_1)<\mathfrak{C}(\lambda_2,s_2)$ even if $\lambda_1<\lambda_2$ just by choosing properly the exponents $s_1>s_2$.

\begin{theorem}\label{thm:groundstatesalphabeta}
Assume \eqref{H1} and either $\frac{\alpha}{2_{s_1}^*}+\frac{\beta}{2_{s_2}^*}< 1$ or $\frac{\alpha}{2_{s_1}^*}+\frac{\beta}{2_{s_2}^*}= 1$ satisfying \eqref{H}. Then, 
\begin{itemize}
\item[$i)$] If $\alpha\ge 2$ and $\mathfrak{C}(\lambda_1,s_1) > \mathfrak{C}(\lambda_2,s_2)$, then there exists $\tilde{\nu}>0$ such that for any $0<\nu<\tilde{\nu}$  the couple $(0,z_\mu^{(2)})$ is the ground state of \eqref{system:alphabeta}.
\item[$ii)$]  If $\beta\ge 2$ and $\mathfrak{C}(\lambda_1,s_1) < \mathfrak{C}(\lambda_2,s_2)$, then there exists $\tilde{\nu}>0$ such that for any $0<\nu<\tilde{\nu}$  the couple $(z_\mu^{(1)},0)$ is the ground state of \eqref{system:alphabeta}.
\item[$iii)$] In particular, if $\alpha,\beta \ge 2$, then there exists $\tilde{\nu}>0$ such that for any $0<\nu<\tilde{\nu}$, the couple $(0, z_\mu^{(2)})$ is a ground state of \eqref{system:alphabeta} if $\mathfrak{C}(\lambda_1,s_1) > \mathfrak{C}(\lambda_2,s_2)$, whereas $(z_\mu^{(1)},0)$ is a ground state otherwise.
\end{itemize}
\end{theorem}
Concerning the bound states, we derive its existence by applying min-max techniques. In fact, $J_\nu$ exhibits a Mountain--Pass geometry for a certain choice of the parameters $s_1,s_2,\lambda_1$ and $\lambda_2$. In particular, if the values $\mathfrak{C}(\lambda_1,s_1)$ and $\mathfrak{C}(\lambda_2,s_2)$ verify a certain separability assumption, we can raise the mountain pass level above that of the \textit{semitrivials} ones.
\begin{theorem}\label{MPgeom}
Assume \eqref{H1} and either $\frac{\alpha}{2_{s_1}^*}+\frac{\beta}{2_{s_2}^*}< 1$ or $\frac{\alpha}{2_{s_1}^*}+\frac{\beta}{2_{s_2}^*}= 1$ satisfying \eqref{H}. If
\begin{itemize}
\item[$i)$] Either $\alpha \ge 2$ and
\begin{equation}\label{lamdasalphabeta}
2\mathfrak{C}(\lambda_2,s_2)> \mathfrak{C}(\lambda_1,s_1)>  \mathfrak{C}(\lambda_2,s_2),
\end{equation}
\item[$ii)$] or $\beta \ge 2 $ and
\begin{equation*}\label{lamdasalphabeta2}
2\mathfrak{C}(\lambda_1,s_1)> \mathfrak{C}(\lambda_2,s_2)>  \mathfrak{C}(\lambda_1,s_1),
\end{equation*}
\end{itemize}
then there exists $\tilde{\nu}>0$ such that for $0<\nu\le \tilde{\nu}$, the system \eqref{system:alphabeta} admits a bound state given as a Mountain--Pass-type critical point.
\end{theorem}

The rest of the paper is organized as follows. In Section~\ref{section2} we introduce some preliminaries concerning the variational structure of the problem and some useful properties regarding the \textit{semitrivial} solutions. In particular, in that section we discuss in detail the role of the potential $h$ and we deduce its belonging to the appropriate $L^{p,\sigma}(\mathbb{R}^N)$ space. In Section~\ref{section:PS} we study the validity of the Palais Smale condition. Finally, the existence of bound and ground states and its symmetry is addressed to Section~\ref{section:main}.

\section{Preliminaries and Functional setting}\label{section2}

Let us introduce the appropriate variational setting for the system \eqref{system:alphabeta}. Recall that such solutions are critical points of the functional $\mathcal{J}_\nu$ introduced in \eqref{functalphabeta} correctly defined in $\mathbb{D}=\mathcal{D}^{1,2} (\mathbb{R}^N)\times \mathcal{D}^{1,2} (\mathbb{R}^N)$. The energy space $\mathbb{D}$ is equipped with the norm
\begin{equation*}
\|(u,v)\|^2_{\mathbb{D}}=\|u\|^2_{\lambda_1}+\|v\|^2_{\lambda_2},
\end{equation*}
where

\begin{equation*}
\|u\|^2_{\lambda}=\int_{\mathbb{R}^N} |\nabla u|^2 \, dx - \lambda \int_{\mathbb{R}^N} \frac{u^2}{|x|^2} \, dx.
\end{equation*}
By applying the Hardy inequality,
\begin{equation*}\label{hardy_inequality}
\Lambda_N \int_{\R^N} \frac{u^2}{|x|^2} \, dx \leq \int_{\R^N} |\nabla u|^2 \, dx,
\end{equation*}
the norm $\|\cdot\|_{\lambda}$ is equivalent to $\|\cdot\|_{\mathcal{D}^{1,2} (\mathbb{R}^N)}$ for any $\lambda\in (0,\Lambda_N)$, where $\Lambda_N=\frac{(N-2)^2}{4}$ is the best constant in the Hardy inequality.

\subsection{The role of the function $h$}\label{roleh}

One of the main difficulties when dealing with critical problems is the lack of compactness and the possible associated concentration phenomena. Before proceeding, let us comment on the main role of the function $h$ in the concentration phenomena which in turn highlights the main contribution of this work. As already mentioned, due to the presence of a singular weight of Hardy--Sobolev type, concentration can only take place at $0$ or $\infty$, so the function $h$ is precisely in charge of avoiding such concentration, namely:
 \begin{center}
 $h$ controls the concentration at $0$ and $\infty$ through its integrability together with its behavior around $0$ and $\infty$.
 \end{center}
The role of the integrability of $h$ becomes clear when one tries to bound the coupling term 
	\[\int_{\mathbb{R}^N}h(x)\frac{|u|^\alpha |v|^\beta}{|x|^{s_3}}dx,\]
by relating it to the critical terms arising from $u$ and $v$. Actually, in view of \eqref{alphabeta} and \eqref{eses}, let us write
\begin{equation}\label{tau}
	s_3=\tau+s_1\frac{\alpha}{2_{s_1}^*}+s_2\frac{\beta}{2_{s_2}^*}
\end{equation}
for some $\tau\ge 0$, so that, using the generalized H\"older inequality with exponents
	\begin{equation*}
	\mathfrak{p}=\frac{1}{1-\frac{\alpha}{2_{s_1}^*}-\frac{\beta}{2_{s_2}^*}}\in(1,\infty],
		\quad
	q=\frac{2^*_{s_1}}{\alpha},
		\quad\text{and}\quad
	r=\frac{2^*_{s_2}}{\beta}
	\end{equation*}
we get 
\begin{align}\label{GenHol}
	\int_{\mathbb{R}^N}h(x)\frac{|u|^\alpha|v|^\beta}{|x|^{s_3}}\ dx=~&\int_{\mathbb{R}^N}\frac{h(x)}{|x|^{\tau}}\bigg(\frac{|u|^{2^*_{s_1}}}{|x|^{s_1}}\bigg)^{\frac{\alpha}{2^*_{s_1}}}\bigg(\frac{|v|^{2^*_{s_2}}}{|x|^{s_2}}\bigg)^{\frac{\beta}{2^*_{s_2}}}\ dx
	\notag
	\\
	\leq
	~&
	\bigg(\int_{\mathbb{R}^N}\bigg(\frac{h(x)}{|x|^{\tau}}\bigg)^\mathfrak{p}\ dx\bigg)^{\frac{1}{\mathfrak{p}}}
	\bigg(\int_{\mathbb{R}^N}\frac{|u|^{2^*_{s_1}}}{|x|^{s_1}}\ dx\bigg)^{\frac{\alpha}{2^*_{s_1}}}
	\bigg(\int_{\mathbb{R}^N}\frac{|v|^{2^*_{s_2}}}{|x|^{s_2}}\ dx\bigg)^{\frac{\beta}{2^*_{s_2}}}.
\end{align}
Thus, the hypothesis
\[\frac{h(x)}{|x|^{\tau}}\in L^{\mathfrak{p}}(\mathbb{R}^N),\]
with $\tau$ defined in \eqref{tau}, appears naturally and allows us to relate the coupling term to the uncoupled critical ones. Therefore, we require
\[h\in L^{\mathfrak{p},\sigma}(\mathbb{R}^N),\]
where 
\begin{equation}\label{sigmadef}
\sigma=\tau\mathfrak{p}= \frac{s_3-\left(s_1\frac{\alpha}{2_{s_1}^*}+s_2\frac{\beta}{2_{s_2}^*}\right)}{1-\left(\frac{\alpha}{2_{s_1}^*}+\frac{\beta}{2_{s_2}^*}\right)}.
\end{equation}

Note also that in the critical regime $\frac{\alpha}{2^*_{s_1}}+\frac{\beta}{2^*_{s_2}}=1$ one has $\mathfrak{p}=\infty$, so the hypothesis \eqref{H1} reduces to the boundedness of $\frac{h(x)}{|x|^{\tau}}$ in $L^{\infty}(\mathbb{R}^N)$-norm. To continue, let us briefly discuss some aspects of the above bound for some particular configurations of the parameters $s_1$, $s_2$ and $s_3$.

\begin{enumerate}
\item $s_1=s_2=s_3=s>0$: In this case, analyzed in \cite{LSO}, equation \eqref{sigmadef} reads 
\[\tau=\frac{s}{\mathfrak{p}} \qquad\text{and}\qquad \mathfrak{p}=\frac{1}{1-\frac{\alpha+\beta}{2_s^*}}\qquad\text{and thus}\qquad h\in L^{\mathfrak{p},s}(\mathbb{R}^N).\]
If  $\alpha+\beta=2_s^*$, then $\sigma=0$ and $\mathfrak{p}=\infty$, so that $h\in L^\infty(\mathbb{R}^N)$.
Let us also stress that  $\mathfrak{p}=\infty$ if and only if  $\tau=0$.
\item $s_1=s_2=s$ and $s_3>0$: In this case 
\[\tau=s_3-\frac{\alpha+\beta}{2_s^*}s\qquad\text{and}\qquad p=\frac{1}{1-\frac{\alpha+\beta}{2_s^*}}.\]
The restriction \eqref{eses} reads now $s_3\ge \frac{\alpha+\beta}{2_s^*}s$. Regarding the integrability of $h$:
\begin{enumerate}
\item If $s_3=\frac{\alpha+\beta}{2_s^*}s$: Then $\tau=0$ while $p<\infty$ if $\alpha+\beta<2_s^*$ and $\mathfrak{p}=\infty$ if $\alpha+\beta=2_s^*$, that is
\begin{enumerate}
\item $h\in L^{\mathfrak{p}}(\mathbb{R}^N)$ if $\alpha+\beta<2_s^*$.
\item $h\in L^{\infty}(\mathbb{R}^N)$ if $\alpha+\beta=2_s^*$.
\end{enumerate}
\item If $s_3>\frac{\alpha+\beta}{2_s^*}$:
\begin{enumerate}
\item If $\alpha+\beta<2_s^*$, we have the hypothesis \eqref{H1} with $\sigma=\tau \mathfrak{p}>0$ and $\mathfrak{p}<\infty$.
\item If  $\alpha+\beta=2_s^*$, we get $\mathfrak{p}=\infty$ while $\tau=s_3-s>0$, so that the integrability hypothesis reads
\begin{equation}\label{cot}
\frac{h(x)}{|x|^{\tau}}\in L^{\infty}(\mathbb{R}^N).
\end{equation}
Roughly speaking: if the coupling term is of order $s_3>s$, the function $h$ has to control such excess, in particular, from \eqref{cot} above,
\begin{equation*}
h(x)\sim |x|^{s_3-s}\quad \text{for}\ x<<1\ \text{or}\ x>>1.
\end{equation*}
In this way,
\begin{equation*}
h(x)\frac{|u|^{\alpha-1} |v|^\beta}{|x|^{s_3}}\sim |x|^{s_3-s}\frac{|u|^{\alpha-1} |v|^\beta}{|x|^{s_3}}=\frac{|u|^{\alpha-1} |v|^\beta}{|x|^{s}},
\end{equation*}
so that $h$ homogenizes the Hardy--Sobolev singularities at the possible concentration points, namely, $0$ and $\infty$. 

\end{enumerate}
The next diagram summarizes the above discussion.

\vspace{0.5cm}

\begin{tikzpicture}[line cap=round,line join=round,>=triangle 45,x=1cm,y=1cm]
%\draw [line width=0pt,dotted] (0,0)-- (2,0);
\draw [line width=0.5pt,dashed] (2,0)-- (6,0);
\draw [line width=0.5pt,dotted] (6,0)-- (10,0);
\begin{scriptsize}

\draw [fill=black] (2,0) circle (1.5pt);
\draw (2.16,-0.3) node {$s_3=\frac{\alpha+\beta}{2_s^*}s$};
\draw (2.16,0.3) node {$h\in L^\mathfrak{p}(\mathbb{R}^N)$};

%\draw [fill=blue] (5,0) circle (1.5pt);
\draw (5,-0.3) node {$\alpha+\beta<2_s^*$};
\draw (4.5,0.3) node {$h\in L^{\mathfrak{p},\sigma}(R^N)$};

%\draw [fill=green] (8,0) circle (1.5pt);
\draw (8,-0.3) node {$\alpha+\beta=2_s^*$};

\draw (8,0.3) node {$\frac{h}{|x|^{\tau}}\in L^{\infty}(\mathbb{R}^N)$};
\end{scriptsize}
\end{tikzpicture}

\end{enumerate}

\vspace{0.5cm}

\item $s_1=s_3=s$ and $s_2>0$: In this case
\[\tau=s\left(1-\frac{\alpha}{2_s^*}\right)-s_2\frac{\beta}{2_{s_2}^*}\qquad\text{and}\qquad \mathfrak{p}=\frac{1}{1-\left(\frac{\alpha}{2_s^*}+\frac{\beta}{2_{s_2}^*}\right)}.\]
The restriction \eqref{eses} establishes now
\[s\left(1-\frac{\alpha}{2_s^*}\right)\ge s_2\frac{\beta}{2_{s_2}^*}.\]
Let us take $s_2^*$ be defined as the extremal case of the above inequality, namely
 \[s_2^*=\frac{Ns(2_s^*-\alpha)}{\beta(N-s)+s(2_s^*-\alpha)}.\]
Note that $s_2^*\ge s$. Furthermore, the following relation also holds,
\[\frac{\alpha}{2_s^*}+\frac{\beta}{2_{s_2^*}^*}=\frac{\alpha+\beta}{2_0^*}+\frac{1}{N}s.\]
In particular, $\frac{\alpha}{2_s^*}+\frac{\beta}{2_{s_2^*}^*}=1$ iff $\alpha+\beta=2_s^*$. Regarding the integrability of $h$:

\begin{enumerate}
\item If $s_2= s_2^*$:\newline We have $\tau=0$ while $\mathfrak{p}<\infty$ if $\frac{\alpha}{2_s^*}+\frac{\beta}{2_{s_2^*}^*}< 1$ and $\mathfrak{p}=\infty$ if $\frac{\alpha}{2_s^*}+\frac{\beta}{2_{s_2^*}^*}=1$, so that
\begin{enumerate}
\item $h\in L^{\mathfrak{p}}(\mathbb{R}^N)$ if $\frac{\alpha}{2_s^*}+\frac{\beta}{2_{s_2^*}^*}<1$.
\item $h\in L^{\infty}(\mathbb{R}^N)$ if $\frac{\alpha}{2_s^*}+\frac{\beta}{2_{s_2^*}^*}=1$.
\end{enumerate}

\item If $s_2<s_2^*$:
\begin{enumerate}
\item If $\frac{\alpha}{2_s^*}+\frac{\beta}{2_{s_2}^*}<1$, we have the hypothesis \eqref{H1} with $\sigma=\tau p>0$ and $\mathfrak{p}<\infty$.
\item If  $\frac{\alpha}{2_s^*}+\frac{\beta}{2_{s_2}^*}=1$, then $\mathfrak{p}=\infty$ (and $s_2^*=s$), while $\tau=\frac{\beta}{2_{s_2}^*}(s-s_2)>0$. In line with \eqref{cot}, we have now 
\[ h(x)\sim |x|^{\frac{\beta}{2_{s_2}^*}(s-s_2)}\ \text{for}\ x<<1\ \text{or}\ x>>1.\]
Therefore,
\[h(x)\frac{|u|^{\alpha-1} |v|^\beta}{|x|^{s}}\sim \frac{|u|^{\alpha-1} |v|^\beta}{|x|^{s-\frac{\beta}{2_{s_2}^*}(s-s_2)}}.\]
In particular, if $\frac{\alpha}{2_s^*}+\frac{\beta}{2_{s_2}^*}=1$, then
\[s-\frac{\beta}{2_{s_2}^*}(s-s_2)=s\frac{\alpha}{2_s^*}+s_2\frac{\beta}{2_{s_2}^*}.\]

In this case $h$ has a double mission: to control two orders of concentration $s$ and $s_2$ at the possible concentration points $0$ and $\infty$. 
\end{enumerate}

The following diagram summarizes the above discussion.

\vspace{0.5cm}

\begin{tikzpicture}[line cap=round,line join=round,>=triangle 45,x=1cm,y=1cm]
\draw [line width=0.5pt,dotted] (0,0)-- (4,0);
\draw [line width=0.5pt,dashed] (4,0)-- (8,0);
%\draw [line width=0pt,dotted] (8,0)-- (10,0);
\begin{scriptsize}
\draw (2,0.3) node {$h\in L^{\mathfrak{p},\sigma}(\mathbb{R}^N)$};
\draw (2,-0.35) node {$\frac{\alpha}{2_s^*}+\frac{\beta}{2_{s_2}^*}<1$};
\draw (6,0.3) node {$\frac{h}{|x|^{\tau}}\in L^{\infty}(R^N)$};
\draw (6,-0.35) node {$\frac{\alpha}{2_s^*}+\frac{\beta}{2_{s_2}^*}=1$};

\draw (8.16,0.3) node {$h\in L^{\mathfrak{p}}(\mathbb{R}^N)$};
\draw [fill=black] (8,0) circle (1.5pt);
\draw (8,-0.3) node {$s_2^*$};
\end{scriptsize}
\end{tikzpicture}

\end{enumerate}

\end{enumerate}

On the other hand, once the optimal summability of the function $h$ has been established, it is also necessary to require the hypothesis \eqref{H}. This ensures that while the functions $u$ and $v$ may concentrate at the points $0$ or $\infty$ the behavior of $h$ around those points cancel out the concentration values arising from such concentration phenomenon (see Lemma \ref{lemcritic}).

\subsection{The scalar equation}

As commented before, the uncoupled equations have been studied extensively nowadays. In particular, if either the system is decoupled, i.e. $\nu=0$, or some component vanishes, $u$ or $v$ satisfies the entire equation
\begin{equation}\label{entire}
-\Delta z - \lambda \frac{z}{|x|^2}=\frac{z^{2^*_s-1}}{|x|^{s}} \qquad\mbox{ and }   \qquad z>0\ \mbox{ in } \mathbb{R}^N \setminus \{0\}.
\end{equation}
A complete classification of \eqref{entire} is given in \cite{KangPeng} where it is proved that, if $\lambda\in\left(0,\Lambda_N\right)$, the solutions of \eqref{entire} are given by
\begin{equation}\label{zeta}
z_{\mu}^{(j)}(x)= \mu^{-\frac{N-2}{2}}z_1^{\lambda_j,s}\left(\frac{x}{\mu}\right) \qquad \mbox{ with } \qquad z_1^{\lambda,s}(x)=\dfrac{A(N,\lambda)^{\frac{N-2}{2(2-s)}}}{|x|^{a_{\lambda}}\left(1+|x|^{(2-s)(1-\frac{2}{N-2}a_{\lambda})}\right)^{\frac{N-2}{2-s}}},
\end{equation}
where $\displaystyle A(N,\lambda)=2(\Lambda_N-\lambda)\frac{N-s}{\sqrt{\Lambda_N}}$,  $a_\lambda=\sqrt{\Lambda_N}-\sqrt{\Lambda_N-\lambda}$ and $\mu>0$ is a scaling factor. 
By direct computation, it holds that
\begin{equation}\label{normcrit}
\displaystyle  \|z_\mu^\lambda\|_{\lambda}^{2} = \|z_{\mu}^{\lambda,s}\|_{2_s^*,s}^{2_s^*}=[ \mathcal{S}(\lambda,s)]^{\frac{N-s}{2-s}},
\end{equation}
where $\mathcal{S}(\lambda,s)$ is given in terms of the Rayleigh quotient
\begin{equation*}
\mathcal{S}(\lambda,s)= \inf_{\substack{u\in \mathcal{D}^{1,2}(\mathbb{R}^N)\\
u\not\equiv0}}\frac{\|u\|^2_{\lambda}}{\|u\|_{2_s^*,s}^{2}}= \frac{\|z_\mu^{\lambda,s}\|^2_{\lambda}}{\|z_{\mu}^{\lambda,s}\|_{2_s^*,s}^{2}}.
\end{equation*}
Due to  \cite[Theorem A]{ChouChu} with $\beta_\lambda=-2 a_\lambda$ and $\alpha_{\lambda,s}=-(2_s^*a_\lambda+s)$, we have   
\begin{equation}\label{Slambda}
\mathcal{S}(\lambda,s)=4(\Lambda_N-\lambda)\frac{N-s}{N-2}\left(\frac{N-2}{2(2-s)\sqrt{\Lambda_N-\lambda}}\frac{2\pi^{\frac{N}{2}}}{\Gamma\left(\frac{N}{2}\right)}\frac{\Gamma^2\left(\frac{N-s}{2-s}\right)}{\Gamma\left(\frac{2(N-s)}{2-s}\right)}\right)^{\frac{2-s}{N-s}}.
\end{equation}

Observe that the constant $\mathcal{S}(\lambda,s)$ is decreasing in both $\lambda$ and $s$, so that $\mathcal{S}(0,0)\ge\mathcal{S}(\lambda,s)$. By definition, $\mathcal{S}(\lambda,s)$ is the best constant for the inequality 
\begin{equation}\label{H-S_lambda}
\mathcal{S}(\lambda,s)\|u\|_{2_s^*,s}^{2} \leq \|u\|_\lambda^2.
\end{equation}
Taking $u(x)=|x|^{a_\lambda}z_1^{\lambda,s}$, the equation \eqref{entire} becomes (in a weak sense) 
\begin{equation*}
-\text{div}(|x|^{-2a_\lambda}\nabla u)=\frac{u^{2_s^*-1}}{|x|^{ 2_s^*a_\lambda+s}}.
\end{equation*}
As a consequence, $\mathcal{S}(\lambda,s)$ is the best constant in the Caffarelli--Kohn--Nirenberg type inequality,
\begin{equation*}
\mathcal{S}(\lambda,s) \left(\int_{\R^N} |x|^{\alpha_{\lambda,s}}|u|^{2_s^*} \, dx \right)^{\frac{2}{2_s^*}}\leq \int_{\R^N} |x|^{\beta_{\lambda}}|\nabla u|^2 \, dx.
\end{equation*}
Clearly, $\mathcal{S}(0,0)=\mathcal{S}$ is the best constant in the Sobolev inequality,
\begin{equation*}
\mathcal{S}\|u\|_{2_0^*,0}^{2} \leq \int_{\R^N} |\nabla u|^2dx.
\end{equation*}
Since $\alpha_{\lambda,0}=-\frac{N}{\sqrt{\Lambda_N}}a_\lambda$, we have $\mathcal{S}(\lambda,0)=\left(1-\frac{\lambda}{\Lambda_N} \right)^{\frac{N-1}{N}}\mathcal{S}$. On the other hand, $\mathcal{S}(0,s)$ is the best constant in the Hardy--Sobolev inequality,
\begin{equation*}
\mathcal{S}(0,s) \|u\|_{2_s^*,s}^{2}\leq \int_{\R^N} |\nabla u|^2 \, dx.
\end{equation*}
Furthermore, one can see that (cf. \cite{ChouChu}), $\displaystyle\lim\limits_{s\to2^-}\mathcal{S}(0,s)=\Lambda_N$. 

The pairs $(z_1,0)$ and $(0,z_2)$ satisfying \eqref{system:alphabeta} will be referred to as \textit{semi-trivial} solutions. Here $z_i$ solves \eqref{entire} with $s=s_i$ and $\lambda=\lambda_i$. Further information is provided in subsection \ref{subsection:semitrivials}.
Our main aim is to look for solutions neither \textit{semi-trivial} nor trivial solutions, i.e., solutions $(u,v)$ such that $u\not\equiv 0$ and $v\not\equiv 0$ in $\mathbb{R}^N$. 

We shall do it by means of variational methods. Let us notice first that the functional $\mathcal{J}_\nu$ is of class $C^1(\mathbb{D},\mathbb{R})$ and it is not bounded from below. Indeed,
\begin{equation*}
\mathcal{J}_\nu(t \tilde u,t \tilde v) \to -\infty \quad \mbox{ as } t\to\infty,
\end{equation*}
for $(\tilde u,\tilde v)\in \mathbb{D}\setminus\{(0,0)\}$. Next, we introduce a suitable constraint to minimize $\mathcal{J}_\nu$. Let us set 
\begin{equation*}
\begin{split}
\Psi(u,v)=&\left\langle \mathcal{J}'_\nu(u,v){\big|}(u,v)\right\rangle\\
=&\|(u,v)\|_\mathbb{D}^2  -\|u\|_{2_{s_1}^*,s_1}^{2_{s_1}^*} - \|v\|_{2_{s_2}^*,s_2}^{2_{s_2}^*} -\nu (\alpha+\beta) \int_{\R^N} h(x) \frac{|u|^{\alpha} |v|^{\beta}}{|x|^{s_3}} \, dx,
\end{split}
\end{equation*}
and define the Nehari manifold associated to $\mathcal{J}_\nu$ as
\begin{equation*}
\mathcal{N}_\nu=\left\{ (u,v) \in \mathbb{D} \setminus \{(0,0)\} \, : \, \Psi(u,v) =0 \right\}.
\end{equation*}
Plainly, $\mathcal{N}_\nu$ contains all the non-trivial critical points of $\mathcal{J}_\nu$ in $\mathbb{D}$. Let us now recall some properties on $\mathcal{N}_\nu$ that will be of use throughout this work. Any $(u,v) \in \mathcal{N}_\nu$ satisfies
\begin{equation} \label{Nnueq1}
 \|(u,v)\|_\mathbb{D}^2=\|u\|_{2_{s_1}^*,s_1}^{2_{s_1}^*}+ \|v\|_{2_{s_2}^*,s_2}^{2_{s_2}^*} +\nu (\alpha+\beta) \int_{\R^N} h(x)  \frac{|u|^{\alpha} |v|^{\beta}}{|x|^{s_3}} \, dx,
\end{equation}
so we can write the energy functional constrained to $\mathcal{N}_\nu$ as
\begin{equation}\label{Nnueq}
\begin{split}
\mathcal{J}_{\nu}{\big|}_{\mathcal{N}_\nu} (u,v) =& \left( \frac{1}{2}-\frac{1}{\alpha+\beta} \right) \|(u,v)\|^2_{\mathbb{D}} \\
&+ \left( \frac{1}{\alpha+\beta}-\frac{1}{2^*_{s_1}} \right)\|u\|_{2_{s_1}^*,s_1}^{2_{s_1}^*}+ \left( \frac{1}{\alpha+\beta}-\frac{1}{2^*_{s_2}} \right)\|v\|_{2_{s_2}^*,s_2}^{2_{s_2}^*} \\
= & \bigg(\frac{1}{2}-\frac{1}{2^*_{s_1}}\bigg) \|u\|_{2_{s_1}^*,s_1}^{2_{s_1}^*}+\bigg(\frac{1}{2}-\frac{1}{2^*_{s_2}}\bigg)\|v\|_{2_{s_2}^*,s_2}^{2_{s_2}^*} +\nu \left( \frac{\alpha+\beta-2}{2} \right)\int_{\R^N} h(x) \frac{|u|^{\alpha} |v|^{\beta}}{|x|^{s_3}} dx.
\end{split}
\end{equation}
Given $(u,v)\in\mathbb{D}\setminus \{(0,0)\}$, there exists a unique value $t=t_{(u,v)}$ such that $(tu,tv) \in \mathcal{N}_\nu$. Indeed, $t$ is the unique solution to the algebraic equation
\begin{equation}\label{normH}
\|(u,v)\|_\mathbb{D}^2=\ t^{2_{s_1}^*-2}\|u\|_{2_{s_1}^*,s_1}^{2_{s_1}^*} + t^{2_{s_2}^*-2} \|v\|_{2_{s_2}^*,s_2}^{2_{s_2}^*} +\nu (\alpha+\beta) \, t^{\alpha+\beta-2} \int_{\R^N} h(x) \frac{ |u|^{\alpha} |v|^{\beta}}{|x|^{s_3}}dx.
\end{equation}
Observe that, by \eqref{alphabeta}, we have
\begin{equation}\label{natural}
\left\langle \Psi'(u,v){\big|}(u,v)\right\rangle<0,
\end{equation}

since, because of \eqref{Nnueq1}, for any $(u,v) \in \mathcal{N}_\nu$,
\begin{equation*}\label{criticalpoint1}
\begin{split}
\left\langle \Psi'(u,v){\big|}(u,v)\right\rangle=&\ (2-\alpha-\beta)\|(u,v)\|_{\mathbb{D}}^2\\
 &+ (\alpha+\beta-2_{s_1}^*)\|u\|_{2_{s_1}^*,s_1}^{2_{s_1}^*} + (\alpha+\beta-2_{s_2}^*) \|v\|_{2_{s_2}^*,s_2}^{2_{s_2}^*} \vspace{0.2cm}\\
=&\ (2-2^*_{s_1}) \|u\|_{2_{s_1}^*,s_1}^{2_{s_1}^*}  +  (2-2^*_{s_2}) \|v\|_{2_{s_2}^*,s_2}^{2_{s_2}^*}\\
& + \nu(\alpha+\beta)(2-\alpha-\beta) \int_{\R^N} h(x) \frac{ |u|^{\alpha} |v|^{\beta}}{|x|^{s_3}}dx.
\end{split}
\end{equation*}
Moreover, the couple $(0,0)$ is a strict minimum since, for the second variation of the energy functional, one has that
\begin{equation*}
 \mathcal{J}_\nu''(0,0)[\varphi_1,\varphi_2]^2=\|(\varphi_1,\varphi_2)\|^2_{\mathbb{D}} \quad \text{ for any } (\varphi_1,\varphi_2)\in \mathcal{N}_\nu.
\end{equation*}
Hence, $(0,0)$ is an isolated point respect to $\displaystyle \mathcal{N}_\nu  \, \cup \, \{(0,0)\}$. As a consequence, $\mathcal{N}_\nu$ is a smooth complete manifold of codimension $1$. Also, there exists  $r_\nu>0$ such that
\begin{equation}\label{criticalpoint2}
\|(u,v)\|_{\mathbb{D}} > r_\nu\quad\text{for all } (u,v)\in \mathcal{N}_\nu.
\end{equation}

By using \eqref{Nnueq1}, \eqref{Nnueq} and \eqref{criticalpoint2} joint with hypothesis \eqref{alphabeta}, one can infer that that 
\begin{equation*}
\mathcal{J}_{\nu} (u,v) > C(r_\nu)>0  \quad\text{for all } (u,v)\in \mathcal{N}_\nu.
\end{equation*}
Thus, $\mathcal{J}_{\nu}$ remains bounded from below on $\mathcal{N}_\nu$ and, hence, we can find solutions of \eqref{system:alphabeta} as minimizers of $\mathcal{J}_{\nu}{\big|}_{\mathcal{N}_\nu}$.\newline

Finally, let us also note that, given $(u,v) \in \mathbb{D}$ a critical point of $\mathcal{J}_{\nu}{\big|}_{\mathcal{N}_\nu}$, there exists a Lagrange multiplier $\omega$ such that
\begin{equation*}
(\mathcal{J}_{\nu}{\big|}_{\mathcal{N}_\nu})'(u,v)=\mathcal{J}'_\nu(u,v)-\omega \Psi'(u,v)=0.
\end{equation*}
Then, $\omega \left\langle \Psi'(u,v){\big|}(u,v)\right\rangle=\left\langle \mathcal{J}'_\nu(u,v){\big|}(u,v)\right\rangle = \Psi(u,v)=0$ so that, because of \eqref{natural}, we have $\omega=0$ and, thus, $\mathcal{J}'_\nu(u,v)=0$. Consequently, $\mathcal{N}_\nu$ is a called a natural constraint in the sense that
\begin{equation*}\label{equivalent}
(u,v) \in \mathbb{D}\quad \text{is a critical point of}\quad \mathcal{J}_\nu \qquad\Leftrightarrow\qquad (u,v) \in \mathbb{D}\quad \text{is a critical point of}\quad \mathcal{J}_{\nu}{\big|}_{\mathcal{N}_\nu}.
\end{equation*}

\begin{definition}
We say that $(u,v)\in\mathbb{D}\setminus \{(0,0)\}$ is a non-trivial bound state for \eqref{system:alphabeta} if it is a non-trivial critical point of $\mathcal{J}_\nu$. 
A non-trivial and non-negative bound state $(\tilde{u},\tilde{v})$ is said to be a ground state if its energy is minimal, namely

\begin{equation}\label{ctilde}
\tilde{c}_\nu\vcentcolon=\mathcal{J}_\nu(\tilde{u},\tilde{v})=\min\left\{\mathcal{J}_\nu(u,v): (u,v)\in \mathcal{N}_\nu,\; u,v\ge 0 \right\}.
\end{equation}

\end{definition}
\subsection{Semi-trivial solutions}\label{subsection:semitrivials}\hfill\newline
Let us consider the decoupled energy functionals $\mathcal{J}_j:\mathcal{D}^{1,2} (\mathbb{R}^N)\mapsto\mathbb{R}$,
\begin{equation}\label{funct:Ji}
\mathcal{J}_j(u) =\frac{1}{2} \int_{\mathbb{R}^N}  |\nabla u|^2 \, dx -\frac{\lambda_j}{2} \int_{\mathbb{R}^N} \dfrac{u^2}{|x|^2}  \, dx - \frac{1}{2^*_{s_j}} \int_{\mathbb{R}^N} \frac{|u|^{2^*_{s_j}}}{|x|^{s_j}} \, dx.
\end{equation}
Notice that
\begin{equation*}
\mathcal{J}_\nu(u,v)=\mathcal{J}_1(u)+\mathcal{J}_2(v)-\nu \int_{\mathbb{R}^N} h(x) \frac{|u|^{\alpha} |v|^{\beta}}{|x|^{s_3}} \, dx.
\end{equation*}
The function $z_{\mu}^{(j)}$, defined in \eqref{zeta}, is a global minimum of $\mathcal{J}_j$ on the Nehari manifold
\begin{equation*}\label{Neharij}
\begin{split}
\mathcal{N}_j&= \left\{ u \in \mathcal{D}^{1,2} (\mathbb{R}^N) \setminus \{0\} \, : \,  \left\langle \mathcal{J}'_j(u){\big|} u\right\rangle=0 \right\}\\
&= \left\{ u \in\mathcal{D}^{1,2} (\mathbb{R}^N) \setminus \{0\} \, : \,  \|u\|_{\lambda_j}^2=\|u\|_{2_{s_j}^*,s_j}^{2_{s_j}^*}\right\}.
\end{split}
\end{equation*}
By using \eqref{normcrit}, one can compute the energy levels of $z_\mu^{(j)}$, namely,  for any $\mu>0$ we have
\begin{equation}\label{critical_levels}
\mathfrak{C}(\lambda_j,s_j)\vcentcolon=\mathcal{J}_j(z_\mu^{(j)})=\dfrac{2-s_j}{2(N-s_j)}\left[\mathcal{S}(\lambda_j,s_j)\right]^{\frac{N-s_j}{2-s_j}}.
\end{equation}
Then, the energy levels of the \textit{semi-trivial} solutions are given by
\begin{equation*}\label{Jzeta}
\mathcal{J}_\nu(z_\mu^{(1)},0)=\mathfrak{C}(\lambda_1,s_1) \qquad\text{and}\qquad \mathcal{J}_\nu(0,z_\mu^{(2)})=\mathfrak{C}(\lambda_2,s_2).
\end{equation*}
Let us remark that, since $\mathcal{S}(\lambda,s)$ is decreasing in both $\lambda$ and $s$, we have 
\begin{equation*}\label{decreasing}
\mathfrak{C}(0,0)\ge\mathfrak{C}(\lambda,s)\qquad\text{for }\lambda\in(0,\Lambda_N)\ \text{and }s\in(0,2).
\end{equation*}

Next, we characterize the variational nature of the \textit{semi-trivial} couples on $\mathcal{N}_\nu$. The proof follows from \cite[Theorem 2.2]{LSO} introducing conveniently the H\"older inequality proved in \eqref{GenHol}.
\begin{theorem}\label{thmsemitrivialalphabeta} Under hypotheses \eqref{alphabeta} and \eqref{H}, the following holds:
\begin{enumerate}
\item[i)] If $\alpha>2$ or $\alpha=2$ and $\nu$ small enough, then $(0,z_\mu^{(2)})$ is a local minimum of $\mathcal{J}_\nu$ on $\mathcal{N}_\nu$.
\item[ii)] If $\beta>2$ or $\beta=2$ and $\nu$ small enough, then $(z_\mu^{(1)},0)$ is a local minimum of $\mathcal{J}_\nu$ on $\mathcal{N}_\nu$.
\item[iii)] If $\alpha<2$ or $\alpha=2$ and $\nu$ large enough, then $(0,z_\mu^{(2)})$ is a saddle point for $\mathcal{J}_\nu$ on $\mathcal{N}_\nu$.
\item[iv)] If $\beta<2$ or $\beta=2$ and $\nu$ large enough, then $(z_\mu^{(1)},0)$ is a saddle point for $\mathcal{J}_\nu$ on $\mathcal{N}_\nu$.
\end{enumerate}
\end{theorem}

To conclude this section, let us introduce an algebraic result which extends \cite[Lemma 3.3]{AbFePe}, corresponding to $s_1=0=s_2$; and \cite[Lemma 2.3]{LSO}, corresponding to $s_1=s_2=s>0$, to our weighted setting dealing with arbitrary $s_1,s_2\in(0,2)$.
\begin{lemma}\label{algelemma}
Let $P, Q,R>0$, $s_1,s_2\in(0,2)$ and $\alpha,\beta$ such that \eqref{alphabeta} is satisfied. Consider the set
\begin{equation*}
\Sigma_\nu=\{\sigma \in (0,+\infty)  \, : \, P \sigma^{\frac{2}{2^*_{s_1}}}+Q \sigma^{\frac{2}{2_{s_2}^*}} < (P+Q)\sigma + R \nu \sigma^{\frac{\alpha}{2^*_{s_1}}+\frac{\beta}{2^*_{s_2}}} \}.
\end{equation*}
Then, for every $\varepsilon>0$ there exists $\tilde{\nu}>0$ such that
\begin{equation*}
\inf_{\Sigma_\nu} \sigma > 1-\varepsilon \qquad \mbox{ for any } 0<\nu<\tilde{\nu}.
\end{equation*}
\end{lemma}

\begin{proof}
We shall consider the function $f:(0,\infty)\to\mathbb{R}$ defined as 
$$
f(\sigma)=P\sigma^{\frac{2-\alpha}{2^*_{s_1}}-\frac{\beta}{2^*_{s_2}}}+Q\sigma^{\frac{2-\beta}{2^*_{s_2}}-\frac{\alpha}{2^*_{s_1}}}-(P+Q)\sigma^{1-\frac{\alpha}{2^*_{s_1}}-\frac{\beta}{2^*_{s_2}}}.
$$

With loss of generality, let us assume that $s_2\ge s_1$. By the assumptions \eqref{alphabeta}, we can deduce that $\frac{2-\alpha}{2^*_{s_1}}-\frac{\beta}{2^*_{s_2}}<0$. Since $f(1)=0$, in the case $\frac{2-\beta}{2^*_{s_2}}-\frac{\alpha}{2^*_{s_1}}\leq 0$, $f$ is decreasing and then $1$ is the unique zero so we arrive at the desired conclusion.

Suppose now $\frac{2-\beta}{2^*_{s_2}}-\frac{\alpha}{2^*_{s_1}}>0$. In particular, it holds $1-\frac{\alpha}{2^*_{s_1}}-\frac{\beta}{2^*_{s_2}}>\frac{2-\beta}{2^*_{s_2}}-\frac{\alpha}{2^*_{s_1}}$. Let us now consider the function 
$$
g(\sigma)=Q\sigma^{\frac{2-\beta}{2^*_{s_2}}-\frac{\alpha}{2^*_{s_1}}}-(P+Q)\sigma^{1-\frac{\alpha}{2^*_{s_1}}-\frac{\beta}{2^*_{s_2}}}.
$$

Observe that $g$ admits a maximum at $\tilde{\sigma}=\frac{P}{P+Q}\frac{\frac{2-\beta}{2^*_{s_2}}-\frac{\alpha}{2^*_{s_1}}}{1-\frac{\alpha}{2^*_{s_1}}-\frac{\beta}{2^*_{s_2}}}<1$ and $g(\tilde{\sigma})>0$. Since $g(\sigma)$ is decreasing if $\sigma>\tilde{\sigma}$, we can deduce that function $f$ is decreasing if $\sigma>\tilde{\sigma}$. Then $1$ is also a unique zero of the function $f$ and the conclusion follows.

\end{proof}

\section{The Palais--Smale condition}\label{section:PS}
As it is customary in critical problems, some care has to be taken when dealing with minimizing sequences as the compactness of the Sobolev embedding does not hold for the critical exponent. This compactes will be ensured by following the nowadays well-known Palais--Smale approach. We start by recalling some definitions. 
\begin{definition}
Let $V$ be a Banach space. We say that $\{u_n\} \subset V$ is a PS sequence at level $c$ for an energy functional $\mathfrak{F}:V\mapsto\mathbb{R}$ if
\begin{equation*}
\mathfrak{F}(u_n) \to c \quad \hbox{ and }\quad  \mathfrak{F}'(u_n) \to 0\quad\mbox{in}\ V'\quad \hbox{as}\quad n\to + \infty,
\end{equation*}
where $V'$ is the dual space of $V$. Moreover, we say that the functional $\mathfrak{F}$ satisfies the PS condition at level $c$ if every PS sequence at $c$ for $\mathfrak{F}$ has a strongly convergent subsequence.
\end{definition}
The proof of the following results is similar to \cite[Lemma 3.2]{CoLSOr} and \cite[Lemma 3.3]{CoLSOr} respectively, so we omit the details.
\begin{lemma}\label{lemma:PSNehari}
Let $\{(u_n,v_n)\} \subset \mathcal{N}_\nu$ be a PS sequence for $\mathcal{J}_\nu {\big|}_{\mathcal{N}_\nu}$ at level $c\in\mathbb{R}$. Then, $\{(u_n,v_n)\}$ is a PS sequence for $\mathcal{J}_\nu$ in $\mathbb{D}$, namely
\begin{equation}\label{PSNehari}
\mathcal{J}_{\nu}'(u_n,v_n)\to0 \quad \mbox{ in } \mathbb{D}' \quad\text{as }n\to+\infty.
\end{equation}
\end{lemma}
\begin{lemma}\label{lemmaPS0}
Let $\{(u_n,v_n)\} \subset \mathbb{D}$ be a PS sequence for the energy functional $\mathcal{J}_\nu$ at level $c\in\mathbb{R}$. Then,  $\|(u_n,v_n)\|_{\mathbb{D}}<C$.
\end{lemma}

Let us point out here that as a consequence of this result, given any PS sequence bounded in $\mathbb{D}$ one can subtract a subsequence $\{(u_n,v_n)\}$ converging to $(\tilde u,\tilde v)\in\mathbb{D}$ in the sense that
\begin{align*}
	&
	(u_n,v_n)
	\rightharpoonup
	(\tilde u,\tilde v)
	\quad\text{weakly in $\mathbb{D}$},
	\\
	&
	(u_n,v_n)
	\to
	(\tilde u,\tilde v)
	\quad\text{strongly in $L^{q_1,s_1}(\R^N)\times L^{q_2,s_2}(\R^N)$ for $1\leq q_i<2^*_{s_i}$ with $i=1,2$},
	\\
	&
	(u_n,v_n)
	\to
	(\tilde u,\tilde v)
	\quad\text{a.e. in $\R^N$}.
\end{align*}
Moreover, by the Concentration-Compactness Principle (cf. \cite{Lions1,Lions2}), we have (up to a subsequence if necessary)
\begin{equation}\label{limits0}
\begin{split}
	\lim_{n\to\infty}\int_{\R^N}|\nabla u_n|^2\varphi\ dx
	\geq
	~&
	\mu_0\varphi(0)+\int_{\R^N}|\nabla\tilde u|^2\varphi\ dx,
	\\
	\lim_{n\to\infty}\int_{\R^N}\frac{u_n^2}{|x|^2}\varphi\ dx
	=
	~&
	\eta_0\varphi(0)+\int_{\R^N}\frac{\tilde u^2}{|x|^2}\varphi\ dx,
	\\
	\lim_{n\to\infty}\int_{\R^N}\frac{|u_n|^{2^*_{s_1}}}{|x|^{s_1}}\varphi\ dx
	=
	~&
	\rho_0\varphi(0)+\int_{\R^N}\frac{|\tilde u|^{2^*_{s_1}}}{|x|^{s_1}}\varphi\ dx,
\end{split}
\end{equation}
for every function $\varphi\geq 0$ decaying to $0$ at infinity, where  $\mu_0,\rho_0,\eta_0>0$ are fixed constants. On the other hand, the concentration at infinity can be described by the constants
\begin{equation}\label{limits8}
\begin{split}
	\mu_\infty
	=
	~&
	\lim_{R\to\infty}\limsup_{n\to\infty}\int_{|x|>R}|\nabla u_n|^2\ dx,
	\\
	\rho_\infty
	=
	~&
	\lim_{R\to\infty}\limsup_{n\to\infty}\int_{|x|>R}\frac{|u_n|^{2^*_{s_1}}}{|x|^{s_1}}\ dx,
	\\
	\eta_\infty
	=
	~&
	\lim_{R\to\infty}\limsup_{n\to\infty}\int_{|x|>R}\frac{u_n^2}{|x|^2}\ dx.
\end{split}
\end{equation}
Analogous limits are also satisfied for the sequence $\{v_n\}$, for which we denote the constants as $\overline\mu_0,\overline\rho_0,\overline\eta_0>0$ for the concentration at $0$ and $\overline\mu_\infty,\overline\rho_\infty,\overline\eta_\infty>0$ at infinity.

Next we check that the Palais--Smale condition is satisfied for certain energy levels in both the subcritical and critical ranges, but before, let us stress that one of the key ideas in both cases is to test $\mathcal{J}_\nu'(u_n,v_n)$ with $(u_n\varphi_\eps,0)$ (resp. $(0,v_n\varphi_\eps)$), where $\varphi_\eps$ is either $\varphi_{0,\eps}$ or $\varphi_{\infty,\eps}$, two smooth cut-off functions supported in a neighborhood of $0$ and near $\infty$, respectively. More precisely,
\begin{equation}\label{tests}
\begin{split}
	\varphi_{0,\eps}(x)
	=
	~&
	\begin{cases}
	1 & \text{ if } |x|<\frac{\eps}{2},\\
	0 & \text{ if } |x|\geq\eps,
	\end{cases}
	\quad\text{ with }\quad
	|\nabla\varphi_{0,\eps}|\leq\frac{4}{\eps},
	\\
	\varphi_{\infty,\eps}(x)
	=
	~&
	\begin{cases}
	0 & \text{ if } |x|<\frac{1}{\eps},\\
	1 & \text{ if } |x|\geq\frac{1}{\eps}+1,
	\end{cases}
	\quad\text{ with }\quad
	|\nabla\varphi_{\infty,\eps}|\leq2.
\end{split}
\end{equation}
Then using the limits in \eqref{limits0} we can estimate the quantity
\begin{align*}
	0
	=
	~&
	\lim_{\eps\to0}\lim_{n\to\infty}\langle\mathcal{J}_\nu'(u_n,v_n)|(u_n\varphi_{0,\eps},0)\rangle
	\\
	=
	~&
	\lim_{\eps\to0}\lim_{n\to\infty}\bigg[
	\int_{\R^N}|\nabla u_n|^2\varphi_{0,\eps}\ dx
	+\int_{\R^N}u_n\nabla u_n\nabla\varphi_{0,\eps}\ dx
	\\
	~&\phantom{\lim_{\eps\to0}\lim_{n\to\infty}\bigg[}
	-\lambda_1\int_{\R^N}\frac{u_n^2}{|x|^2}\varphi_{0,\eps}\ dx
	-\int_{\R^N}\frac{|u_n|^{2^*_{s_1}}}{|x|^{s_1}}\varphi_{0,\eps}\ dx 	-\nu\alpha\int_{\R^N}h(x)\frac{|u_n|^\alpha|v_n|^\beta}{|x|^{s_3}}\varphi_{0,\eps}\ dx
	\bigg]
	\\
	\geq
	~&
	\mu_0-\lambda_1\eta_0-\rho_0
	\\
	~&
	+\lim_{\eps\to0}\bigg[
	\int_{\R^N}|\nabla\tilde u|^2\varphi_{0,\eps}\ dx
	+\int_{\R^N}\tilde u\nabla\tilde u\nabla\varphi_{0,\eps}\ dx
	\\
	~&\phantom{+\lim_{\eps\to0}\bigg[}
	-\lambda_1\int_{\R^N}\frac{\tilde u^2}{|x|^2}\varphi_{0,\eps}\ dx
	-\int_{\R^N}\frac{|\tilde u|^{2^*_{s_1}}}{|x|^{s_1}}\varphi_{0,\eps}\ dx	-\nu\alpha\lim_{n\to\infty}\int_{\R^N}h(x)\frac{|u_n|^\alpha|v_n|^\beta}{|x|^{s_3}}\varphi_{0,\eps}\ dx
	\bigg]
	\\
	=
	~&
	\mu_0-\lambda_1\eta_0-\rho_0-\nu\alpha\lim_{\eps\to0}\lim_{n\to\infty}\int_{\R^N}h(x)\frac{|u_n|^\alpha|v_n|^\beta}{|x|^{s_3}}\varphi_{0,\eps}\ dx.
\end{align*}

Observe that in the subcritical range, one can deduce immediately that the limit in the last line above vanishes. This is also the case in the critical range, but we have to be more careful with the estimates. This is done in Lemma~\ref{lemcritic} below.

Analogously for $\varphi_{\infty,\eps}$ using \eqref{limits8} we have
\begin{align*}
	0
	=
	~&
	\lim_{\eps\to0}\lim_{n\to\infty}\langle\mathcal{J}_\nu'(u_n,v_n)|(u_n\varphi_{\infty,\eps},0)\rangle
	\\
	\geq
	~&
	\mu_\infty-\lambda_1\eta_\infty-\rho_\infty-\nu\alpha\lim_{\eps\to0}\lim_{n\to\infty}\int_{\R^N}h(x)\frac{|u_n|^\alpha|v_n|^\beta}{|x|^{s_3}}\varphi_{\infty,\eps}\ dx,
\end{align*}
and the last limit vanishes in both the subcritical and the critical range.

Combining this with the inequality \eqref{H-S_lambda} we get
\begin{equation}\label{cotas}
\begin{split}
	\mathcal{S}(\lambda_1,s_1)\rho_0^{\frac{2}{2^*_{s_1}}}
	\leq
	~&
	\mu_0-\lambda_1\eta_0
	\leq
	\rho_0,
	\\
	\mathcal{S}(\lambda_1,s_1)\rho_\infty^{\frac{2}{2^*_{s_1}}}
	\leq
	~&
	\mu_\infty-\lambda_1\eta_\infty
	\leq
	\rho_\infty,
\end{split}
\qquad\qquad
\begin{split}
	\mathcal{S}(\lambda_2,s_2)\overline\rho_0^{\frac{2}{2^*_{s_2}}}
	\leq
	~&
	\overline\mu_0-\lambda_2\overline\eta_0
	\leq
	\overline\rho_0,
	\\
	\mathcal{S}(\lambda_2,s_2)\overline\rho_\infty^{\frac{2}{2^*_{s_2}}}
	\leq
	~&
	\overline\mu_\infty-\lambda_2\overline\eta_\infty
	\leq
	\overline\rho_\infty.
\end{split}
\end{equation}

\subsection{Subcritical range $ \frac{\alpha}{2_{s_1}^*}+\frac{\beta}{2_{s_2}^*} <1 $}\hfill\newline
We establish now the Palais--Smale condition for subcritical energy levels of $\mathcal{J}_\nu$, which will allow us to find existence of solutions for \eqref{system:alphabeta} by means of minimizing sequences. 
\begin{lemma}\label{lemmaPS2}
 Assume that $\frac{\alpha}{2_{s_1}^*}+\frac{\beta}{2_{s_2}^*}<1$ and hypothesis \eqref{eses} holds.  Then, the functional $\mathcal{J}_\nu$ satisfies the PS condition for any level $c$ such that
\begin{equation}\label{hyplemmaPS2}
c<\min\left\{ \mathfrak{C}(\lambda_1,s_1), \mathfrak{C}(\lambda_2,s_2) \right\}.
\end{equation}
\end{lemma}

\begin{proof}
Let $\{u_n,v_n\}\subset \mathcal{N}_\nu$ be a PS sequence for $\mathcal{J}_\nu|_{\mathcal{N}_\nu}$ at level $c$. Then, by using \eqref{Nnueq}, one has 

\begin{equation}\label{clim}
\begin{split}
	c	&=
	\lim_{n\to\infty}\Bigg[	\left(\frac{1}{2}-\frac{1}{2^*_{s_1}} \right) \|u_n\|_{2_{s_1}^*,s_1}^{2_{s_1}^*} + \left(\frac{1}{2}-\frac{1}{2^*_{s_2}} \right) \|v_n\|_{2_{s_2}^*,s_2}^{2_{s_2}^*} + \nu \left(\frac{\alpha+\beta}{2}-1\right) \int_{\R^N} h(x) \frac{ |u_n|^{\alpha} |v_n|^{\beta}}{|x|^{s_3}}dx 	\Bigg] \\
& \geq \lim_{n\to\infty}\Bigg[ \left(\frac{1}{2}-\frac{1}{2^*_{s_1}} \right) \|u_n\|_{2_{s_1}^*,s_1}^{2_{s_1}^*} + \left(\frac{1}{2}-\frac{1}{2^*_{s_2}} \right) \|v_n\|_{2_{s_2}^*,s_2}^{2_{s_2}^*} 	\Bigg].
\end{split}
\end{equation}
Recalling \eqref{limits0} and \eqref{limits8},
\begin{equation}\label{clim2}
	c
	\geq
	\bigg(\frac{1}{2}-\frac{1}{2^*_{s_1}}\bigg)\Big[\|\tilde u\|_{2^*_{s_1},s_1}^{2^*_{s_1}}+\rho_0+\rho_\infty\Big]
	+
	\bigg(\frac{1}{2}-\frac{1}{2^*_{s_2}}\bigg)\Big[\|\tilde v\|_{2^*_{s_2},s_2}^{2^*_{s_2}}+\overline\rho_0+\overline\rho_\infty\Big].
\end{equation}

From this it can be deduced that $\rho_0=\overline\rho_0=\rho_\infty=\overline\rho_\infty=0$. On the contrary, let us assume for example that $\rho_0>0$. With the aid of \eqref{cotas}, we can estimate
\begin{equation}\label{cotas2}
	\rho_0
	\geq
	\mathcal{S}(\lambda_1,s_1)^{\frac{2^*_{s_1}}{2^*_{s_1}-2}},
\end{equation}
so replacing this above and estimating the rest of the terms by $0$ we get
\begin{equation*}
	c
	\geq
	\bigg(\frac{1}{2}-\frac{1}{2^*_{s_1}}\bigg)\rho_0
	\geq
	\bigg(\frac{1}{2}-\frac{1}{2^*_{s_1}}\bigg)\mathcal{S}(\lambda_1,s_1)^{\frac{2^*_{s_1}}{2^*_{s_1}-2}}
	=
	\mathfrak{C}(\lambda_1,s_1),
\end{equation*}
which contradicts the assumption \eqref{hyplemmaPS2}, so $\rho_0=0$. By the same argument, $\rho_\infty=\overline\rho_0=\overline\rho_\infty=0$. Then one can deduce the existence of a strongly convergent subsequence in $L^{2^*_{s_1},s_1}(\R^N)\times L^{2^*_{s_2},s_2}(\R^N)$ such that
\begin{equation*}
	\lim_{n\to\infty}\Big[\|(u_n-\tilde u,v_n-\tilde v)\|_{\mathbb{D}}^2
	-\left\langle\mathcal{J}_\nu'(u_n,v_n)\big|(u_n-\tilde u,v_n-\tilde v)\right\rangle\Big]
	=
	0,
\end{equation*}
and thus the PS condition is satisfied.
\end{proof}

\

Now, we shall study the validity of PS condition for supercritical energy levels. First, let us introduce the truncated problem
\begin{equation}\label{systemp}
\left\{\begin{array}{ll}
\displaystyle -\Delta u - \lambda_1 \frac{u}{|x|^2}-\frac{(u^+)^{2_{s_1}^*-1}}{|x|^{s_1}}=  \nu\alpha h(x) \frac{(u^+)^{\alpha-1}\, (v^+)^{\beta}}{|x|^{s_3}}  &\text{in }\mathbb{R}^N,\vspace{.3cm}\\
\displaystyle -\Delta v - \lambda_2 \frac{v}{|x|^2}-\frac{(v^+)^{2_{s_2}^*-1}}{|x|^{s_2}}= \nu \beta h(x)  \frac{(u^+)^{\alpha}\, (v^+)^{\beta-1}}{|x|^{s_3}} &\text{in }\mathbb{R}^N,
\end{array}\right.
\end{equation}
where  $u^+=\max\{u,0\}$. Note that $u=u^+ + u^-$ where $u^-$ is negative part of the function $u$, i.e., $u^-=\min\{u,0\}$. Observe that a solution of \eqref{system:alphabeta} satisfies \eqref{systemp}.

Such a system admits a variational structure and its associated energy functional is\begin{equation*}\label{funct:SKdVp}
\mathcal{J}^+_\nu (u,v)= \|(u,v)\|^2_{\mathbb{D}}- \frac{1}{2^*_{s_1}} \|u^+\|_{2_{s_1}^*,s_1}^{2_{s_1}^*} - \frac{1}{2^*_{s_2}} \|v^+\|_{2_{s_2}^*,s_2}^{2_{s_2}^*} -\nu \int_{\mathbb{R}^N} h(x) \frac{(u^+)^\alpha (v^+)^{\beta}}{|x|^{s_3}} dx,
\end{equation*}
defined in $\mathbb{D}$. As we did with $\mathcal{J}_\nu$ we can define the Nehari manifold related to $\mathcal{J}^+_\nu $. Actually,
\begin{equation*}
\mathcal{N}^+_\nu=\left\{ (u,v) \in \mathbb{D} \setminus \{(0,0)\} \, : \,  \left\langle (\mathcal{J}_\nu^+)' (u,v){\big|}(u,v)\right\rangle=0 \right\}.
\end{equation*}

\begin{lemma}\label{lemmaPS1}
Assume that $\frac{\alpha}{2_{s_1}^*}+\frac{\beta}{2_{s_2}^*}<1$, $\alpha\ge2$ and $\mathfrak{C}(\lambda_1,s_1)\geq \mathfrak{C}(\lambda_2,s_2)$. Then, there exists $\tilde{\nu}>0$ such that, if $0<\nu\leq\tilde{\nu}$ and $\{(u_n,v_n)\} \subset \mathbb{D}$ is a PS sequence for $\mathcal{J}^+_\nu$ at level $c\in\mathbb{R}$ such that
\begin{equation}\label{PS1}
\mathfrak{C}(\lambda_1,s_1)<c<\mathfrak{C}(\lambda_1,s_1)+\mathfrak{C}(\lambda_2,s_2)
\end{equation}
and
\begin{equation*}\label{PS2}
c\neq \ell \mathfrak{C}(\lambda_2,s_2) \quad \mbox{ for every } \ell \in \mathbb{N}\setminus \{0\},
\end{equation*}
then $(u_n,v_n)\to(\tilde{u},\tilde{v}) \in \mathbb{D}$ up to subsequence.
\end{lemma}

\begin{proof}
Our aim is to show that
\begin{equation}\label{claimPSmain}
	u_n\to \tilde u\  \mbox{ in } \mathcal{D}^{1,2}(\R^N)
	\qquad \mbox{and} \qquad
	v_n\to \tilde v\ \mbox{ in } \mathcal{D}^{1,2}(\R^N).
\end{equation}
In order to do that, we first show that at least one of these sequences converges in $\mathcal{D}^{1,2}(\R^N)$. Then, in a second step, distinguishing cases we prove that the convergence in $\mathcal{D}^{1,2}(\R^N)$ of one of the sequences implies the convergence of the other, from which \eqref{claimPSmain} follows.

To see that \eqref{claimPSmain} holds for at least one of the sequences, we start by claiming that
\begin{equation}\label{claimPS21}
	\mbox{ either }
	u_n\to \tilde u\  \mbox{  strongly in } L^{2^*_{s_1},s_1}(\mathbb{R}^N)
	\qquad \mbox{or} \qquad
	v_n\to \tilde v\ \mbox{  strongly in }L^{2^*_{s_2},s_2}(\mathbb{R}^N).
\end{equation}
Indeed, if any of these sequences converges strongly in its corresponding $L^{2^*_{s_i},s_i}(\R^N)$, say $\{u_n\}_n$, then
\begin{equation*}
	\lim_{n\to\infty}\|u_n-\tilde u\|_{\lambda_1}^2
	=
	\lim_{n\to\infty}\langle\mathcal{J}_\nu'(u_n,v_n)|(u_n-\tilde u,0)\rangle,
\end{equation*}
from which by the equivalence of the norms $\|\cdot\|_\lambda$ and $\|\cdot\|_{\mathcal{D}^{1,2}(\R^N)}$ it follows that $\{u_n\}_n$ converges in $\mathcal{D}^{1,2}(\R^N)$ as desired.

To check that \eqref{claimPS21} holds, let us assume on the contrary that both $\{u_n\}$ and $\{v_n\}$ do not converge strongly in $L^{2^*_{s_1}}(\R^N)$ and $L^{2^*_{s_2}}(\R^N)$, respectively. Then we have $\rho_j>0$ and $\overline\rho_k>0$ for some $j,k\in\{0,\infty\}$. Repeating the argument leading to \eqref{clim2} in the previous lemma, estimating by $0$ and recalling \eqref{cotas} we obtain
\begin{align*}
	c
	\geq
	~&
		\bigg(\frac{1}{2}-\frac{1}{2^*_{s_1}}\bigg)\rho_j
	+
	\bigg(\frac{1}{2}-\frac{1}{2^*_{s_2}}\bigg)\overline\rho_k
	\\
	\geq
	~&
	\bigg(\frac{1}{2}-\frac{1}{2^*_{s_1}}\bigg)\mathcal{S}(\lambda_1,s_1)^{\frac{2^*_{s_1}}{2^*_{s_1}-2}}
	+
	\bigg(\frac{1}{2}-\frac{1}{2^*_{s_2}}\bigg)\mathcal{S}(\lambda_2,s_2)^{\frac{2^*_{s_2}}{2^*_{s_2}-2}}
	\\
	=
	~&
	\mathfrak{C}(\lambda_1,s_1)+\mathfrak{C}(\lambda_2,s_2),
\end{align*}
which is in contradiction with the hypothesis \eqref{PS1}, and so \eqref{claimPS21} follows.

At this point, we have already shown that at least one of the sequences in \eqref{claimPSmain} converges in $\mathcal{D}^{1,2}(\R^N)$. To show that \eqref{claimPSmain} actually holds simultaneously for both sequences we discard the cases in which one of the sequences does not converge.

\begin{enumerate}
\item Suppose that $u_n\not\to\tilde u$ and $v_n\to\tilde v$ in $\mathcal{D}^{1,2}(\R^N)$. Arguing as in the proof of the previous lemma, we recall \eqref{clim2} to write
\begin{equation}\label{clim3}
	c
	\geq
	\bigg(\frac{1}{2}-\frac{1}{2^*_{s_1}}\bigg)\Big[\|\tilde u\|_{2^*_{s_1},s_1}^{2^*_{s_1}}+\rho_0+\rho_\infty\Big].
\end{equation}

If $\{u_n\}$ concentrates simultaneously at $0$ and $\infty$, estimating by $0$ the norm of $\tilde u$ and using the inequalities in \eqref{cotas} we get
\begin{equation}\label{clim4}
	c
	\geq
	2\bigg(\frac{1}{2}-\frac{1}{2^*_{s_1}}\bigg)\mathcal{S}(\lambda_1,s_1)^{\frac{2^*_{s_1}}{2^*_{s_1}-2}}
	=
	2\mathfrak{C}(\lambda_1,s_1)
	\geq
	\mathfrak{C}(\lambda_1,s_1)+\mathfrak{C}(\lambda_2,s_2),
\end{equation}
where in the last inequality we have used that $\mathfrak{C}(\lambda_1,s_1)\geq\mathfrak{C}(\lambda_2,s_2)$ by assumption. Then we get a contradiction with the condition \eqref{PS1}. Thus $\{u_n\}$ does not concentrate simultaneously at $0$ and $\infty$.

Next, we show that $\tilde v$ is not identically zero by contradiction. Assume on the contrary that $\tilde v\equiv 0$. Then it holds that either $\widetilde u\not\equiv0$ or $\widetilde u\equiv 0$. In the first case, if $\widetilde u\geq0$, then $\tilde u=z_\mu^{(1)}$ for some $\mu>0$, and so, by \eqref{normcrit},
\begin{equation*}
	\|\tilde u\|_{2^*_{s_1},s_1}^{2^*_{s_1}}
	=
	\mathcal{S}(\lambda_1,s_1)^{\frac{2^*_{s_1}}{2^*_{s_1}-2}}.
\end{equation*}
Since either $\rho_0=0$ or $\rho_\infty=0$, replacing this identity in \eqref{clim3} it turns out that the inequalities in \eqref{clim4} also hold, so we reach again a contradiction with the condition \eqref{PS1}. On the other hand, if $\tilde u\equiv 0$, since $\{u_n\}$ concentrates at most at one point, then
\begin{equation*}
	c
	=
	\lim_{n\to\infty}\mathcal{J}_\nu(u_n,v_n)
	=
	\left(\frac{1}{2}-\frac{1}{2^*_{s_1}}\right)\|u_n\|_{2_{s_1}^*,s_1}^{2_{s_1}^*}
	=
	\left(\frac{1}{2}-\frac{1}{2^*_{s_1}} \right)\rho_j,
\end{equation*}
with $j=0$ or $j=\infty$, so $\{u_n\}$ is a positive PS sequence for $\mathcal{J}_j(\tilde u)$ (see \eqref{funct:Ji}). Thus by \cite[Theorem 3.1]{Li},
\begin{equation*}
	\rho_j
	=
	\ell\mathcal{S}(\lambda_1,s_1)^{\frac{2^*_{s_1}}{2^*_{s_1}-2}},
\end{equation*}
for some $\ell\in\N$, so $c=\ell\mathfrak{C}(\lambda_1,s_1)$, which is in contradiction with the assumptions $c>\mathfrak{C}(\lambda_1,s_1)\geq\mathfrak{C}(\lambda_2,s_2)$. Indeed, if $\ell=1$ then the contradiction is clear, on the other hand, if $\ell\geq 2$ then $c\geq\mathfrak{C}(\lambda_1,s_1)+\mathfrak{C}(\lambda_2,s_2)$, which contradicts \eqref{PS1}. Hence $\tilde v$ is not identically zero in $\R^N$.

Next we show that the sequence $\{u_n\}$ converges weakly in $\mathcal{D}^{1,2}(\R^N)$ to $\tilde u\not\equiv 0$. For that, if we assume on the contrary that $\tilde u\equiv 0$, repeating the argument above we find that $\tilde v=z_\mu^{(2)}$ for some $\mu>0$ and the contradiction follows analogously. Thus $\tilde u$ and $\tilde v$ are not identically zero in $\R^N$, so recalling the identity in \eqref{clim} and the fact that $\{u_n\}$ concentrates at most at one point it turns out that
\begin{equation}\label{etiqueta}
\begin{split}
	\mathfrak{C}(\lambda_1,s_1)+\mathfrak{C}(\lambda_2,s_2)
	\geq
	c
	=
	~&
	\left(\frac{1}{2}-\frac{1}{2^*_{s_1}}\right)\Big[\|\tilde u\|_{2_{s_1}^*,s_1}^{2_{s_1}^*}+\rho_j\Big]
	+
	\left(\frac{1}{2}-\frac{1}{2^*_{s_2}}\right)\|\tilde v\|_{2_{s_2}^*,s_2}^{2_{s_2}^*}
	\\
	~&
	+
	\nu\left(\frac{\alpha+\beta}{2}-1\right)\int_{\R^N}h(x)\frac{\tilde u^{\alpha}\tilde v^{\beta}}{|x|^{s_3}}\ dx,
\end{split}
\end{equation}
where $j=0$ or $j=\infty$. Observe that, as a consequence, since $\displaystyle\lim_{n\to\infty}\left\langle\mathcal{J}'_\nu(u_n,v_n){\big|}(\tilde u,\tilde v)\right\rangle=0$ is equivalent to $(\tilde u,\tilde v)\in\mathcal{N}_\nu$, by \eqref{Nnueq} we have
\begin{equation*}
	\mathcal{J}_\nu(\tilde u,\tilde v)
	=
	c-\left(\frac{1}{2}-\frac{1}{2^*_{s_1}}\right)\rho_j.
\end{equation*}
Using the assumption \eqref{PS1} and the inequalities in \eqref{cotas} as in \eqref{cotas2}, as well as the definition of $\mathfrak{C}(\lambda_1,s_1)$ in \eqref{critical_levels}, we obtain
\begin{equation*}
	\mathcal{J}_\nu(\tilde u,\tilde v)
	<
	\mathfrak{C}(\lambda_1,s_1)
	+
	\mathfrak{C}(\lambda_2,s_2)
	-
	\left(\frac{1}{2}-\frac{1}{2^*_{s_1}}\right)\mathcal{S}(\lambda_1,s_1)^{\frac{2^*_{s_1}}{2^*_{s_1}-2}}
	=
	\mathfrak{C}(\lambda_2,s_2).
\end{equation*}
where we have used \eqref{cotas}. As a consequence,
\begin{equation*}
	\tilde c_\nu
	=
	\inf_{(u,v)\in\mathcal{N}_\nu}\mathcal{J}_\nu(u,v)
	<
	\mathfrak{C}(\lambda_2,s_2),
\end{equation*}
which in view of Theorem~\ref{thm:groundstatesalphabeta} contradicts the fact that $\tilde c_\nu=\mathfrak{C}(\lambda_2,s_2)$ for sufficiently small $\nu$. Then $u_n\to\tilde u$ strongly in $\mathcal{D}^{1,2}(\R^N)$ as desired.

\item Suppose that $u_n\to\tilde u$ and $v_n\not\to\tilde v$ in $\mathcal{D}^{1,2}(\R^N)$. First we claim that both $\tilde u$ and $\tilde v$ are not identically zero. In that case, recalling the identity in \eqref{clim} as in the previous case we have
\begin{align*}
	c
	=
	~&
	\left(\frac{1}{2}-\frac{1}{2^*_{s_1}}\right)\|\tilde u\|_{2_{s_1}^*,s_1}^{2_{s_1}^*}
	+
	\left(\frac{1}{2}-\frac{1}{2^*_{s_2}}\right)\Big[\|\tilde v\|_{2_{s_2}^*,s_2}^{2_{s_2}^*}+\overline\rho_0+\overline\rho_\infty\Big]
	\\
	~&
	+
	\nu\left(\frac{\alpha+\beta}{2}-1\right)\int_{\R^N}h(x)\frac{\tilde u^{\alpha}\tilde v^{\beta}}{|x|^{s_3}}\ dx,
\end{align*}
where at least one of the constants $\overline\rho_0$ and $\overline\rho_\infty$ is strictly positive. Again by \eqref{Nnueq},
\begin{equation*}
	\mathcal{J}_\nu(\tilde u,\tilde v)
	=
	c-\left(\frac{1}{2}-\frac{1}{2^*_{s_2}}\right)(\overline\rho_0+\overline\rho_\infty)
	\leq
	c-\left(\frac{1}{2}-\frac{1}{2^*_{s_2}}\right)\overline\rho_j,
\end{equation*}
where $j=0$ or $j=\infty$, so by \eqref{PS1}, the inequalities in \eqref{cotas} and the definition of $\mathfrak{C}(\lambda_2,s_2)$ in \eqref{critical_levels}, we obtain
\begin{equation*}\label{Jc-ineq}
	\mathcal{J}_\nu(\tilde u,\tilde v)
	<
	\mathfrak{C}(\lambda_1,s_1)
	+
	\mathfrak{C}(\lambda_2,s_2)
	-
	\left(\frac{1}{2}-\frac{1}{2^*_{s_2}}\right)\mathcal{S}(\lambda_2,s_2)^{\frac{2^*_{s_2}}{2^*_{s_2}-2}}
	\leq
	\mathfrak{C}(\lambda_1,s_1).
\end{equation*}

Considering the first equation in \eqref{system:alphabeta} together with the definition of \eqref{H-S_lambda} we deduce that
\begin{align*}
	\mathcal{S}(\lambda_1,s_1)\|\tilde u\|_{2^*_{s_1},s_1}^2
	\leq
	~&
	\|\tilde u\|_{2^*_{s_1},s_1}^{2^*_{s_1}}
	+
	\nu\int_{\R^N}h(x)\frac{\tilde u^\alpha\tilde v^\beta}{|x|^{s_3}}\ dx
	\\
	\leq
	~&
	\|\tilde u\|_{2^*_{s_1},s_1}^{2^*_{s_1}}
	+
	\nu\|h\|_{\mathfrak{p},\sigma}\|\tilde u\|_{2^*_{s_1},s_1}^\alpha\|\tilde v\|_{2^*_{s_2},s_2}^\beta,
\end{align*}
where in the second inequality we have used \eqref{GenHol} with $\mathfrak{p}$ and $\sigma$ as in \eqref{sigmadef}. To estimate the norm of $\tilde v$, we observe that  the inequality \eqref{PS1} and \eqref{etiqueta} imply that \break $\mathcal{J}_\nu(\tilde u,\tilde v)<\mathfrak{C}(\lambda_1,s_1)+\mathfrak{C}(\lambda_2,s_2)\leq2\mathfrak{C}(\lambda_1,s_1)$, and by \eqref{Nnueq} we can bound
\begin{equation*}
	\left(\frac{1}{2}-\frac{1}{2^*_{s_2}}\right)\|\tilde v\|_{2^*_{s_2},s_2}^{2^*_{s_2}}
	\leq
	\mathcal{J}_\nu(\tilde u,\tilde v)
	<
	2\mathfrak{C}(\lambda_1,s_1)
	=
	2\left(\frac{1}{2}-\frac{1}{2^*_{s_1}}\right)\mathcal{S}(\lambda_1,s_1)^{\frac{2^*_{s_1}}{2^*_{s_1}-2}},
\end{equation*}
so we obtain
\begin{equation*}
	\|\tilde v\|_{2^*_{s_2},s_2}^\beta
	<
	\left(2\,\frac{\frac{1}{2}-\frac{1}{2^*_{s_1}}}{\frac{1}{2}-\frac{1}{2^*_{s_2}}}\right)^{\frac{\beta}{2^*_{s_2}}}\mathcal{S}(\lambda_1,s_1)^{\frac{\beta}{2^*_{s_2}}\cdot\frac{2^*_{s_1}}{2^*_{s_1}-2}}.
\end{equation*}
Replacing above,
\begin{equation*}
	\mathcal{S}(\lambda_1,s_1)\|\tilde u\|_{2^*_{s_1},s_1}^2
	\leq
	\|\tilde u\|_{2^*_{s_1},s_1}^{2^*_{s_1}}
	+
	C\nu\|\tilde u\|_{2^*_{s_1},s_1}^\alpha,
\end{equation*}
for some constant $C>0$ depending on $h$, $s_1$, $s_2$ and $\beta$. On the other hand, since $\tilde v$ is not identically zero, it turns out that
\begin{equation*}
	\left(\frac{1}{2}-\frac{1}{2^*_{s_2}}\right)\|\tilde v\|_{2^*_{s_2},s_2}^{2^*_{s_2}}
	\geq
	\tilde\eps
	>
	0.
\end{equation*}
Therefore, if we take a sufficiently small $\eps>0$ such that $\eps\mathfrak{C}(\lambda_1,s_1)\leq\tilde\eps$, then by Lemma~\ref{algelemma}, there exists $\tilde\nu>0$ in such a way that
\begin{equation*}
	\|\tilde u\|_{2^*_{s_1},s_1}
	\geq
	(1-\eps)\mathcal{S}(\lambda_1,s_1)^{\frac{N-s_1}{2-s_1}}
\end{equation*}
for any $0<\nu\leq\tilde\nu$. Putting all these estimates together we get
\begin{equation*}
	\mathcal{J}_\nu(\tilde u,\tilde v)
	\geq
	(1-\eps)\bigg(\frac{1}{2}-\frac{1}{2^*_{s_1}}\bigg)\mathcal{S}(\lambda_1,s_1)^{\frac{N-s_1}{2-s_1}}+\tilde\eps
	=
	\mathfrak{C}(\lambda_1,s_1).
\end{equation*}
Observe that this is a contradiction with the inequality $\mathcal{J}_\nu(\tilde u,\tilde v)<\mathfrak{C}(\lambda_1,s_1)$ obtained above. Thus $v_n\to\tilde v$ strongly in $\mathcal{D}^{1,2}(\R^N)$.
\end{enumerate}

The above discussion shows that both $u_n\to\tilde u$ and $v_n\to\tilde v$ in $\mathcal{D}^{1,2}(\R^N)$, which is precisely our main claim \eqref{claimPSmain}, so the proof is concluded.
\end{proof}

The following can be shown in an analogous way.

\begin{lemma}\label{lemmaPS1b}
Assume that $\frac{\alpha}{2_{s_1}^*}+\frac{\beta}{2_{s_2}^*}<1$, $\beta\ge2$ and $\mathfrak{C}(\lambda_2,s_2)\ge\mathfrak{C}(\lambda_1,s_1)$. Then, there exists $\tilde{\nu}>0$ such that, if $0<\nu\leq\tilde{\nu}$ and $\{(u_n,v_n)\} \subset \mathbb{D}$ is a PS sequence for $\mathcal{J}^+_\nu$ at level $c\in\mathbb{R}$ such that
\begin{equation}\label{PS1b}
\mathfrak{C}(\lambda_2,s_2)<c<\mathfrak{C}(\lambda_1,s_1)+\mathfrak{C}(\lambda_2,s_2)
\end{equation}
and
\begin{equation}\label{PS2b}
c\neq \ell \mathfrak{C}(\lambda_1,s_1) \quad \mbox{ for every } \ell \in \mathbb{N}\setminus \{0\},
\end{equation}
then $(u_n,v_n)\to(\tilde{u},\tilde{v}) \in \mathbb{D}$ up to subsequence.
\end{lemma}

\subsection{Critical range $\frac{\alpha}{2_{s_1}^*}+\frac{\beta}{2_{s_2}^*}=1$}

For the sake of simplicity, let us write $\widetilde h(x)=\frac{h(x)}{|x|^{\sigma}}$. Hence, in view of \eqref{sigmadef}, $\widetilde h\in L^\infty(\R^N)$. Recall also that $\lim_{x\to+\infty}h(x)=0$ by assumption and $\sigma\geq0$, so $\lim_{x\to+\infty}\widetilde h(x)=0$. However, the continuity of $h$ at $0$ with $h(0)=0$ is not enough for our purposes in this case, and we have to impose the continuity of $\widetilde h$ at $0$ with $\widetilde h(0)=0$. In short, we have to assume that the function $h$ satisfies \eqref{H}, namely
\[\widetilde h(x)=\frac{h(x)}{|x|^{\sigma}}\in L^\infty(\R^N) \text{ is continuous at $0$ and $\infty$ and } \widetilde h(0)=0 \mbox{ and } \lim_{x\to+\infty}\widetilde h(x)=0.\]

\begin{lemma}\label{lemcritic}
Assume that $\frac{\alpha}{2_{s_1}^*}+\frac{\beta}{2_{s_2}^*}=1$ and hypothesis \eqref{eses} holds. In addition, let us assume that $\tilde h$ satisfies \eqref{H} for $\sigma$ is as in \eqref{sigmadef}. Let $\{(u_n,v_n)\}\subset\mathbb{D}$ be a PS sequence for $\mathcal{J}_\nu$ at level $c\in\R$ such that
\begin{enumerate}
\item either 
$c<\min\left\{ \mathfrak{C}(\lambda_1,s_1), \mathfrak{C}(\lambda_2,s_2) \right\}$,
\item or if $\alpha\geq2$ and $\mathfrak{C}(\lambda_1,s_1)\ge\mathfrak{C}(\lambda_2,s_2)$ then $\mathfrak{C}(\lambda_1,s_1)<c<\mathfrak{C}(\lambda_1,s_1)+\mathfrak{C}(\lambda_2,s_2)$ and \break $c\neq\ell\mathfrak{C}(\lambda_2,s_2)$ for every $\ell\in\N$,
\item or if $\beta\geq2$ and $\mathfrak{C}(\lambda_1,s_1)\le\mathfrak{C}(\lambda_2,s_2)$ then satifies $\mathfrak{C}(\lambda_2,s_2)<c<\mathfrak{C}(\lambda_1,s_1)+\mathfrak{C}(\lambda_2,s_2)$ and $c\neq\ell\mathfrak{C}(\lambda_1,s_1)$ for every $\ell\in\N$.
\end{enumerate}
Then, there exists $\tilde\nu>0$ such that, for every $0<\nu\leq\tilde\nu$, the sequence $(u_n,v_n)\to(\tilde u,\tilde v)\in\mathbb{D}$ up to a subsequence.
\end{lemma}

\begin{proof}
As noted at the beginning of this section, we need to show that
\begin{align}\label{LimLimsupInt}
	\lim_{\eps\to0}\limsup_{n\to+\infty}\int_{\R^N}h(x)\frac{|u_n|^\alpha|v_n|^\beta}{|x|^{s_3}}\varphi_\eps\ dx
	=
	0
\end{align}
for both $\varphi_\eps=\varphi_{0,\eps}$ and $\varphi_\eps=\varphi_{\infty,\eps}$ the test functions defined in \eqref{tests}, from which immediately follows that the coupling term does not concentrate mass at $0$ and $\infty$.

Use the precise value of $\sigma$ given by \eqref{sigmadef} and apply H\"older inequality with the conjugate exponents $\frac{2^*_{s_1}}{\alpha}$ and $\frac{2^*_{s_2}}{\beta}$ to obtain
\begin{align*}
	\int_{\R^N}h(x)\frac{|u_n|^\alpha|v_n|^\beta}{|x|^{s_3}}\varphi_\eps\ dx
	=
	~&
	\int_{\R^N}
	\bigg(\widetilde h(x)\,\frac{|u_n|^{2^*_{s_1}}}{|x|^{s_1}}\bigg)^{\frac{\alpha}{2^*_{s_1}}}
	\bigg(\widetilde h(x)\,\frac{|v_n|^{2^*_{s_2}}}{|x|^{s_2}}\bigg)^{\frac{\alpha}{2^*_{s_2}}}\varphi_\eps\ dx
	\\
	\leq
	~&
	\bigg(\int_{\R^N}\widetilde h(x)\,\frac{|u_n|^{2^*_{s_1}}}{|x|^{s_1}}\varphi_\eps\ dx\bigg)^{\frac{\alpha}{2^*_{s_1}}}
	\bigg(\int_{\R^N}\widetilde h(x)\,\frac{|v_n|^{2^*_{s_2}}}{|x|^{s_2}}\varphi_\eps\ dx\bigg)^{\frac{\beta}{2^*_{s_2}}}.
\end{align*}
where we have used that $\widetilde h(x)=\frac{h(x)}{|x|^{\sigma}}$. Next we focus on the first integral in the right hand side (the other can be estimated analogously). First, for $\varphi_\eps=\varphi_{0,\eps}$ we use that $\widetilde h$ is continuous at $0$ with $\widetilde h(0)=0$ by assumption together with the third limit in \eqref{limits0} to obtain
\begin{align*}
	\lim_{n\to+\infty}\int_{\R^N}\widetilde h(x)\,\frac{|u_n|^{2^*_{s_1}}}{|x|^{s_1}}\varphi_{0,\eps}\ dx
	=
	\int_{\R^N}\widetilde h(x)\,\frac{|\tilde u|^{2^*_{s_1}}}{|x|^{s_1}}\varphi_{0,\eps}\ dx
	\leq
	\int_{|x|<\eps}\widetilde h(x)\,\frac{|\tilde u|^{2^*_{s_1}}}{|x|^{s_1}}\ dx.
\end{align*}
Similarly for $\varphi_\eps=\varphi_{\infty,\eps}$, we can directly estimate
\begin{align*}
	\lim_{n\to+\infty}\int_{\R^N}\widetilde h(x)\,\frac{|u_n|^{2^*_{s_1}}}{|x|^{s_1}}\varphi_{\infty,\eps}\ dx
	\leq
	~&
	\sup_{|x|>\frac{1}{\eps}}\widetilde h(x)\ \lim_{n\to+\infty}\int_{|x|>\frac{1}{\eps}}\frac{|u_n|^{2^*_{s_1}}}{|x|^{s_1}}\ dx
	\\
	=
	~&
	\rho_\infty\sup_{|x|>\frac{1}{\eps}}\widetilde h(x).
\end{align*}
Thus, since $\widetilde h$ satisfies the conditions in \eqref{H}, taking limits as $\eps\to0$ we obtain \eqref{LimLimsupInt} for both $\varphi_{0,\eps}$ and $\varphi_{\infty,\eps}$.
\end{proof}

\section{Proofs of main Results}\label{section:main}

After analyzing the validity of the Palais--Smale condition, we are in position to prove our main results concerning the existence of solutions for the system \eqref{system:alphabeta}. Let us start by the case of ground states for large coupling parameters.

\begin{proof}[Proof of Theorem~\ref{thm:nugrande}]
Consider $(u,v)\in\mathbb{D}\setminus \{(0,0)\}$. Then there exists a positive constant $t$ such that $(tu,tv) \in \mathcal{N}_\nu$ where $t$ satisfies the equation \eqref{normH}.

Due to $\alpha+\beta>2$ and $2^*_{s_i}>2$, then $t=t_\nu \to 0$ as $\nu\to+\infty$. Moreover, one can obtain the relationship
\begin{equation*}
\lim_{\nu \to+\infty} t_\nu^{\alpha+\beta-2} \nu = \dfrac{\|(u,v)\|_\mathbb{D}^2}{ \displaystyle(\alpha+\beta) \int_{\mathbb{R}^N} h(x) \dfrac{|u|^{\alpha} |v|^{\beta}}{|x|^{s_3}}  \, dx},
\end{equation*}
which implies that the energy of  $(t_\nu u ,t_\nu v)$ is
\begin{equation*}
\mathcal{J}_\nu(t_\nu u ,t_\nu v)=\left(\frac{1}{2}-\frac{1}{\alpha+\beta} +o(1)  \right) t_\nu^2 \|(u,v)\|^2_\mathbb{D}.
\end{equation*}
Then, we point out that
\begin{equation}\label{minimumlevel}
\tilde{c}_\nu=\inf_{(u,v) \in \mathcal{N}_\nu} \mathcal{J}_\nu (u,v)< \min \{\mathcal{J}_\nu(z_\mu^{(1)},0),\mathcal{J}_\nu(0,z_\mu^{(2)}) \}=\min \{ \mathfrak{C}(\lambda_1,s_1), \mathfrak{C}(\lambda_2,s_2)\},
\end{equation}
for some $\nu>\overline{\nu}$ with $\overline{\nu}$ sufficiently large. In the subcritical regime $\frac{\alpha}{2_{s_1}^*}+\frac{\beta}{2_{s_2}^*}<1$, the existence of the minimizer $(\tilde{u},\tilde{v}) \in \mathbb{D}$ such that $\mathcal{J}_\nu(\tilde{u},\tilde{v})=\tilde{c}_\nu$ is a consequence of Lemma~\ref{lemmaPS2}. Concerning the positivity of the solution, it suffices to notice that
\begin{equation*}
\mathcal{J}_\nu(|\tilde{u}|,|\tilde{v}|)= \mathcal{J}_\nu(\tilde{u},\tilde{v}).
\end{equation*}
This enables us to assume that $\tilde{u}\ge 0$ and $\tilde{v}\ge 0$ in $\mathbb{R}^N$. Indeed $\tilde{u}$ and $\tilde{v}$ are smooth in $\R^N\setminus\{0\}$ by classical regularity arguments. 
Let us suppose that, $\tilde{u}\not \equiv 0$ and $\tilde{v}\not \equiv 0$. By contradiction, if $\tilde{u}\equiv 0$, then $\tilde{v}\ge 0 $ and $\tilde{v}$ verifies \eqref{entire}, so $\tilde{v}=z_\mu^{(2)}$, which contradicts the energy level assumption \eqref{minimumlevel}. Analogously, we deduce $\tilde{v}\not\equiv 0$. Finally, by applying the maximum principle in $\R^N\setminus\{0\}$, we get that $(\tilde{u},\tilde{v}) \in \mathcal{N}_\nu$ with $\tilde{u}> 0$ and $\tilde{v}> 0$ in $\mathbb{R}^N\setminus\{0\}$, giving the desired conclusion.
 If $\frac{\alpha}{2_{s_1}^*}+\frac{\beta}{2_{s_2}^*}=1$, one derives the same thesis by applying Lemma~\ref{lemcritic} in a suitable way.

\

Finally, we shall assume that function $h$ is radial and non-increasing. In order to prove the radial symmetry of the ground state, let us set $(\ss{u},\ss{v})$ the Schwartz symmetrization of $(\tilde{u},\tilde{v})$. Using the standard properties concerning symmetric-decreasing rearrangements (see Chapter~3 in \cite{LiebLoss}), we can determine that
\begin{equation}\label{schwartz1}
\begin{split}
\|(\ss{u},\ss{v})\|_{\mathbb{D}} &\le \|(\tilde{u},\tilde{v})\|_{\mathbb{D}}, \vspace{0.3cm} \\
 \|\ss{u}\|^{2^*_{s_1}}_{2^*_{s_1},s_1} + \|\ss{v}\|^{2^*_{s_2}}_{2^*_{s_2},s_2} + \nu (\alpha+\beta)  \int_{\mathbb{R}^N} h^* \dfrac{\ss{u}^\alpha \ss{v}^\beta}{|x|^{s_3}}  & \geq \|\tilde{u}\|^{2^*_{s_1}}_{2^*_{s_1},s_1} + \|\tilde{v}\|^{2^*_{s_2}}_{2^*_{s_2},s_2} + \nu (\alpha+\beta)  \int_{\mathbb{R}^N} h(x) \dfrac{\tilde{u}^{\alpha} \tilde{v}^{\beta}}{|x|^{s_3}} 
\end{split}
\end{equation}
where we have used that that $h^*=h$ by our assumptions. Observe that a priori we do not know if $(\ss{u},\ss{v})$ belongs to $\mathcal{N}_\nu$. For that reason, consider $t_*>0$ such that $t_*(\ss{u},\ss{v})\in \mathcal{N}_\nu$ with $t$ satisfying \eqref{normH}. Since $(\tilde{u},\tilde{v})$ is a critical point of $\mathcal{J}_\nu$ we obtain that
\begin{equation}\label{schwartz2}
0=\left\langle \mathcal{J}'_\nu(\tilde u,\tilde  v){\big|}(\tilde u,\tilde v)\right\rangle \geq \|(\ss{u},\ss{v})\|^2_{\mathbb{D}} -  \|\ss{u}\|^{2^*_{s_1}}_{2^*_{s_1},s_1} - \|\ss{v}\|^{2^*_{s_2}}_{2^*_{s_2},s_2} - \nu (\alpha+\beta)  \int_{\mathbb{R}^N} h(x) \dfrac{\ss{u}^\alpha \ss{v}^\beta}{|x|^{s_3}},
\end{equation}
applying inequalities from \eqref{schwartz1}. Now introducing \eqref{normH} in \eqref{schwartz2}, one has that
\begin{equation*}
\begin{split}
&t_*^{2^*_{s_1}-2} \|\ss{u}\|^{2^*_{s_1}}_{2^*_{s_1},s_1}+t_*^{2^*_{s_2}-2} \|\ss{v}|^{2^*_{s_2}}_{2^*_{s_2},s_2}+\nu (\alpha+\beta)  t_*^{\alpha+\beta-2} \int_{\mathbb{R}^N} h(x) \dfrac{\ss{u}^\alpha \ss{v}^\beta}{|x|^{s_3}} \\
  \leq & \|\ss{u}\|^{2^*_{s_1}}_{2^*_{s_1},s_1}+ \|\ss{v}\|^{2^*_{s_2}}_{2^*_{s_2},s_2}+\nu (\alpha+\beta)   \int_{\mathbb{R}^N} h(x) \dfrac{\ss{u}^\alpha \ss{v}^\beta}{|x|^{s_3}},
\end{split}
\end{equation*}
which obviously implies that $0<t_*\leq1$. Now, combining these expressions, we can estimate the energetic level of $t_*(\ss{u},\ss{v})$. Due to \eqref{normH}, we can write $\mathcal{J}_\nu$ as
\begin{equation*}\label{schwartz3}
\begin{split}
\mathcal{J}_\nu(t_*(\ss{u},\ss{v})) =& \left( \frac{1}{2}-\frac{1}{2^*_{s_1}} \right)t_*^{2_{s_1}^*} \|\ss{u}\|_{2_{s_1}^*,s_1}^{2_{s_1}^*}+ \left( \frac{1}{2}-\frac{1}{2^*_{s_2}} \right)t_*^{2_{s_2}^*}\|\ss{v}\|_{2_{s_2}^*,s_2}^{2_{s_2}^*} \\
&+ \nu \left(\frac{\alpha+\beta}{2}-1\right)   t_*^{\alpha+\beta}\int_{\mathbb{R}^N} h(x) \dfrac{\ss{u}^\alpha \ss{v}^\beta}{|x|^{s_3}} \, dx \\
\leq & \left( \frac{1}{2}-\frac{c}{\alpha+\beta} \right) \|(\tilde{u},\tilde{v})\|^2_{\mathbb{D}} + \left( \frac{c}{\alpha+\beta}-\frac{1}{2^*_{s_1}} \right) \|\tilde{u}\|_{2_{s_1}^*,s_1}^{2_{s_1}^*}\\
&+ \left( \frac{c}{\alpha+\beta}-\frac{1}{2^*_{s_2}} \right)\|\tilde{v}\|_{2_{s_2}^*,s_2}^{2_{s_2}^*} + \nu (c-1)   \int_{\mathbb{R}^N} h(x) \dfrac{\tilde{u}^{\alpha} \tilde{v}^{\beta}}{|x|^{s_3}} \, dx \\
&=\mathcal{J}_\nu(\tilde{u},\tilde{v}),
\end{split}
\end{equation*}
which implies that
\begin{equation*}
\mathcal{J}_\nu(t_*(\ss{u},\ss{v})) \leq \mathcal{J}_\nu(\tilde{u},\tilde{v})=\tilde{c}_\nu.
\end{equation*}
This proves that the ground state is radially symmetric with respect to the origin and non-increasing.
\end{proof}

\begin{proof}[Proof of Theorem~\ref{thm:lambdaground}]
Let us discuss the existence of a positive ground state in case that assumption $i)$ holds. If one suppose $ii)$, the proof is analogous. By Proposition~\ref{thmsemitrivialalphabeta}, the couple $(z_\mu^{(1)},0)$ is a saddle point of $\mathcal{J}_\nu$ on $\mathcal{N}_\nu$. Moreover,
$$
\tilde{c}_\nu<\mathcal{J}_\nu(z_\mu^{(1)},0)=\min \{\mathfrak{C}(\lambda_1,s_1),\mathfrak{C}(\lambda_2,s_2)\},
$$
where $\tilde{c}_\nu $ defined in \eqref{ctilde}. In the subcritical $\frac{\alpha}{2_{s_1}^*}+\frac{\beta}{2_{s_2}^*}<1$, Lemma~\ref{lemmaPS2} ensures the existence  of $(\tilde{u},\tilde{v})\in\mathcal{N}_\nu$ with $\tilde{c}_\nu=\mathcal{J}_\nu(\tilde{u},\tilde{v})$. By employing the approach in the preceding theorem, it follows that $(\tilde{u},\tilde{v})$ is a positive ground state of the system \eqref{system:alphabeta}.
 Regarding the critical regime, the same conclusion is obtained by using Lemma~\ref{lemcritic}.

The radial symmetry of the ground state follows from the approach from the previous proof.

\end{proof}

\begin{proof}[Proof of Theorem~\ref{thm:groundstatesalphabeta}]
We shall prove the result by assuming $i)$. The proof is analogous for the hypotheses $ii)$ and $iii)$. Due to the Proposition~\ref{thmsemitrivialalphabeta}, $(0,z_\mu^{(2)})$ is a local minimum for $\nu$ sufficiently small. By contradiction, we suppose that there exists $\{\nu_n\} \searrow 0$ such that $\tilde{c}_{\nu_n} <  \mathcal{J}_{\nu_n} (0,z_\mu^{(2)})$ where $\tilde{c}_{\nu}$ is defined \eqref{ctilde}. In particular, we have
\begin{equation}\label{groundstates1}
\tilde{c}_{\nu_n}< \min \{ \mathfrak{C}(\lambda_1,s_1), \mathfrak{C}(\lambda_2,s_2) \}= \mathfrak{C}(\lambda_2,s_2).
\end{equation}
Due to Lemma~\ref{lemmaPS2}, the PS condition is satisfied at level $\tilde{c}_{\nu_n}$ for the subcritical regime. If $\frac{\alpha}{2_{s_1}^*}+\frac{\beta}{2_{s_2}^*}=1$, we infer the same deduction by Lemma \ref{lemcritic} with $\nu$ sufficiently small.

For both cases, one obtains the existence of  $(\tilde{u}_n,\tilde{v}_n) \in \mathbb{D}$ such that $\tilde{c}_{\nu_n}=\mathcal{J}_{\nu_n} (\tilde{u}_n,\tilde{v}_n)$. Moreover, $\mathcal{J}_{\nu_n} (\tilde{u}_n,\tilde{v}_n)=\mathcal{J}_{\nu_n} (|\tilde{u}_n|,|\tilde{v}_n|)$, then $\tilde{u}_n \ge 0$ and $\tilde{v}_n \ge 0$. Indeed, as we proceeded in previous results, we can guarantee that $\tilde{u}_n>0$ and $\tilde{v}_n>0$ in $\mathbb{R}^N\setminus \{0\}$.

Due to the expression \eqref{Nnueq}, we have
\begin{equation}\label{groundstates3}
\begin{split}
	\tilde{c}_{\nu_n}
	=
	\mathcal{J}_{\nu_n} (\tilde{u}_n,\tilde{v}_n)
	=
	~&
	\bigg(\frac{1}{2}-\frac{1}{2^*_{s_1}}\bigg)\|\tilde{u}_n\|_{2_{s_1}^*,s_1}^{2_{s_1}^*} +\bigg(\frac{1}{2}-\frac{1}{2^*_{s_2}}\bigg)  \|\tilde{v}_n\|_{2_{s_2}^*,s_2}^{2_{s_2}^*}
	\\
	~&
	+ \nu_n \left( \frac{\alpha+\beta}{2}-1  \right)\int_{\mathbb{R}^N} h(x)  \,  \dfrac{\tilde{u}_n^{\alpha} \,  \tilde{v}_n^\beta } {|x|^{s_3}} \, dx  .
\end{split}
\end{equation}
By using \eqref{groundstates1} and \eqref{groundstates3}, one has that
\begin{equation}\label{groundstates4}
\bigg(\frac{1}{2}-\frac{1}{2^*_{s_1}}\bigg)  \|\tilde{u}_n\|_{2_{s_1}^*,s_1}^{2_{s_1}^*} + \bigg(\frac{1}{2}-\frac{1}{2^*_{s_2}}\bigg)  \|\tilde{v}_n\|_{2_{s_2}^*,s_2}^{2_{s_2}^*} <\mathfrak{C}(\lambda_2,s_2)=\bigg(\frac{1}{2}-\frac{1}{2^*_{s_2}}\bigg) \left[\mathcal{S}(\lambda_2,s_2)\right]^{\frac{N-s_2}{2-s_2}}.
\end{equation}
Note that $(\tilde{u}_n,\tilde{v}_n)$ is a solution of \eqref{system:alphabeta}. Then, by combining its first equation and \eqref{H-S_lambda}, we have
\begin{equation*}\label{groundstates45}
\mathcal{S}(\lambda_1,s_1) \|\tilde{u}_n\|_{2_{s_1}^*,s_1}^{2_{s_1}^*\frac{N-2}{N-s_1}} \leq \|\tilde{u}_n\|_{2_{s_1}^*,s_1}^{2_{s_1}^*} +  \nu_n \alpha  \int_{\mathbb{R}^N} h(x)  \dfrac{\tilde{u}_n^{\alpha} \,  \tilde{v}_n^\beta } {|x|^{s_3}} \, dx  .
\end{equation*}
By the general H\"older's inequality given in \eqref{GenHol}, one gets
\begin{equation*}
\int_{\mathbb{R}^N} h(x)  \,  \dfrac{\tilde{u}_n^{\alpha} \,  \tilde{v}_n^\beta } {|x|^{s}}   \, dx  \leq \|h\|_{p,\sigma}  \|\tilde{u}_n\|_{2_{s_1}^*,s_1}^{\alpha}\|\tilde{v}_n\|_{2_{s_2}^*,s_2}^{\beta},
\end{equation*}
where $\mathfrak{p},\sigma$ were introduced in \eqref{sigmadef}, and then
\begin{equation*}
\int_{\mathbb{R}^N} h(x)  \,  \dfrac{\tilde{u}_n^{\alpha} \,  \tilde{v}_n^\beta } {|x|^{s_3}}   \, dx    \leq C  \|\tilde{u}_n\|_{2_{s_1}^*,s_1}^{\alpha\frac{2_{s_1}^*}{2}\frac{N-2}{N-s_1}} [\mathcal{S}(\lambda_2,s_2)]^{\beta \frac{N-2}{2(2-s_2)}}.
\end{equation*}
Therefore, one obtains that,
\begin{equation*}
\mathcal{S}(\lambda_1,s_1) \|\tilde{u}_n\|_{2_{s_1}^*,s_1}^{2_{s_1}^*\frac{N-2}{N-s_1}} < \|\tilde{u}_n\|_{2_{s_1}^*,s_1}^{2_{s_1}^*} + C  \nu_n\alpha   \|\tilde{u}_n\|_{2_{s_1}^*,s_1}^{\alpha\frac{2_{s_1}^*}{2}\frac{N-2}{N-s_1}} [\mathcal{S}(\lambda_2,s_2)]^{\beta \frac{N-2}{2(2-s_2)}} .
\end{equation*}
Now, let us employ Lemma~\ref{algelemma} with $s_1=s_2$, $P=Q$ and $\sigma=[\mathcal{S}(\lambda_1,s_1)]^{\frac{N-s_1}{s_1-2}} \|\tilde{u}_n\|_{2_{s_1}^*,s_1}^{2_{s_1}^*}$, in order to have $\tilde{\nu}=\tilde{\nu}(\varepsilon)>0$ such that
\begin{equation}\label{groundstates6}
\|\tilde{u}_n\|_{2_{s_1}^*,s_1}^{2_{s_1}^*}> (1-\varepsilon) [\mathcal{S}(\lambda_1,s_1)]^{\frac{N-s_1}{2-s_1}} \qquad \mbox{ for any } 0<\nu_n<\tilde{\nu}.
\end{equation}

Since we have assumed that $\mathfrak{C}(\lambda_1,s_1)>\mathfrak{C}(\lambda_2,s_2)$, one can take some $\varepsilon>0$ with
\begin{equation*}
(1-\varepsilon) \dfrac{2-s_1}{2(N-s_1)}[ \mathcal{S}(\lambda_1,s_1)]^{\frac{N-s_1}{2-s_1}} \ge \dfrac{2-s_2}{2(N-s_2)} [\mathcal{S}(\lambda_2,s_2)]^{\frac{N-s_2}{2-s_2}}.
\end{equation*}

This inequality combined with \eqref{groundstates6} gives us that $\dfrac{2-s_1}{2(N-s_1)}\|\tilde{u}_n\|_{2_{s_1}^*,s_1}^{2_{s_1}^*}>\mathfrak{C}(\lambda_2,s_2)$, which violates \eqref{groundstates4}. Therefore, for $\nu$ sufficiently small it is satisfied that
\begin{equation}\label{groundstates7}
\tilde{c}_\nu =   \bigg(\frac{1}{2}-\frac{1}{2^*_{s_2}}\bigg) [\mathcal{S}(\lambda_2,s_2)]^{\frac{N-s_2}{2-s_2}}.
\end{equation}
Let $(\tilde{u},\tilde{v})$ be a minimizer of $\mathcal{J}_\nu$. Arguing by contradiction, we can state either $\tilde{u}\equiv 0$ or $\tilde{v}\equiv 0$. Actually, if $v\equiv 0$, we contradict assumption \eqref{groundstates7}. So one can deduce that $u\equiv 0$ with $\tilde{v}$ satisfying the equation
\begin{equation*}
-\Delta \tilde{v} - \lambda_2 \frac{\tilde{v}}{|x|^2}=\dfrac{|\tilde{v}|^{2^*_{s_2}-2}\tilde{v}}{|x|^{s_2}} \qquad \mbox{ in } \mathbb{R}^N.
\end{equation*}
To finish, we show that $\tilde{v}= \pm z_{\mu}^{(2)}$. Suppose by contradiction that $\tilde{v}$ changes sign so $\tilde{v}^{\pm} \not \equiv 0$ in $\mathbb{R}^N$. Since $(0,\tilde{v}) \in \mathcal{N}_\nu$, then $(0,\tilde{v}^\pm) \in \mathcal{N}_\nu$ and, by \eqref{groundstates3}, we reach a contradiction, namely,
\begin{equation*}
\tilde{c}_{\nu}= \mathcal{J}_\nu (0,\tilde{v}) = \bigg(\frac{1}{2}-\frac{1}{2^*_{s_2}}\bigg)\|\tilde{v}\|_{2_{s_2}^*,s_2}^{2_{s_2}^*} = \bigg(\frac{1}{2}-\frac{1}{2^*_{s_2}}\bigg)\left( \|\tilde{v}^+\|_{2_{s_2}^*,s_2}^{2_{s_2}^*}  + \|\tilde{v}^-\|_{2_{s_2}^*,s_2}^{2_{s_2}^*}\right)  > \mathcal{J}_\nu (0,\tilde{v}^+) \ge  \tilde{c}_{\nu}.
\end{equation*}
Then, $(0,\pm z_{\mu}^{(2)})$ is the minimizer of $\mathcal{J}_\nu$ in $\mathcal{N}_\nu$ if $\mathfrak{C}(\lambda_1,s_1)>\mathfrak{C}(\lambda_2,s_2)$. Finally, we can state that, under our hypotheses, $(0,z_{\mu}^{(2)})$ is a ground state to \eqref{system:alphabeta}. 
\end{proof} 

\begin{proof}[Proof of Theorem~\ref{MPgeom}]
Let us prove the thesis assuming condition $i)$, as the proof  follows analogously under hypothesis $ii)$. First, we shall  prove that the energy functional  $\mathcal{J}_\nu^+\Big|_{\mathcal{N}^+_\nu} $ admits a Mountain--Pass geometry. Secondly, we show that the PS condition holds for the Mountain--Pass level. As a consequence, we deduce the existence of $(\tilde{u},\tilde{v})\in\mathbb{D}$ which is a critical point of $\mathcal{J}_\nu^+$ and, therefore, a bound state of \eqref{system:alphabeta}.

\textbf{Step 1:} Let us define the set of paths that connects $(z_{\mu}^{(1)},0)$ to $(0,z_{\mu}^{(2)})$ continuously,
\begin{equation*}
\Psi_\nu = \left\{ \psi(t)=(\psi_1(t),\psi_2(t))\in C^0([0,1],\mathcal{N}^+_\nu): \, \psi(0)=(z_1^{(1)},0) \mbox{ and } \, \psi(1)=(0,z_1^{(2)})\right\},
\end{equation*}
and the Mountain--Pass level
\begin{equation*}
c_{MP} = \inf_{\psi\in\Psi_\nu} \max_{t\in [0,1]} \mathcal{J}^+_{\nu} (\psi(t)).
\end{equation*}
Take $\psi=(\psi_1,\psi_2) \in \Psi_\nu$, then by the identity \eqref{Nnueq1}, we obtain that
\begin{equation}\label{MPgeomp1}
\|(\psi_1(t),\psi_2(t))\|^2_{\mathbb{D}}= \|\psi_1^+(t)\|_{2_{s_1}^*,s_1}^{2_{s_1}^*}+\|\psi_2^+(t)\|_{2_{s_2}^*,s_2}^{2_{s_2}^*}  + \nu(\alpha+\beta) \int_{\mathbb{R}^N} h(x) \frac{(\psi_1^+(t))^\alpha (\psi_2^+(t))^\beta}{|x|^{s_3}} \, dx  ,
\end{equation}
and, using \eqref{Nnueq},
\begin{equation}\label{MPgeomp2}
\begin{split}
\mathcal{J}^+_{\nu} (\psi(t)) =&\ \bigg(\frac{1}{2}-\frac{1}{2^*_{s_1}}\bigg)\|\psi_1^+(t)\|_{2_{s_1}^*,s_1}^{2_{s_1}^*} +\bigg(\frac{1}{2}-\frac{1}{2^*_{s_2}}\bigg) \|\psi_2^+(t)\|_{2_{s_2}^*,s_2}^{2_{s_2}^*} \\
  &+ \nu\left(\frac{\alpha+\beta}{2}-1\right) \int_{\mathbb{R}^N}h(x)\frac{(\psi_1^+(t))^\alpha (\psi_2^+(t))^\beta}{|x|^{s_3}} \, dx.
\end{split}
\end{equation}
Let us take $\sigma(t)=\left(\sigma_1(t),\sigma_2(t)\right)$, with $\displaystyle\sigma_j(t)\vcentcolon=\|\psi_j^+(t)\|_{2_{s_j}^*,s_j}^{2_{s_j}^*}$. Then, by \eqref{H-S_lambda} and \eqref{MPgeomp1}, 
\begin{equation}\label{MPgeomp3}
\begin{split}
\mathcal{S}(\lambda_1,s_1)(\sigma_1(t))^{\frac{N-2}{N-s_1}}+\mathcal{S}(\lambda_2,s_2)(\sigma_2(t))^{\frac{N-2}{N-s_2}}
\leq&\ \|(\psi_1(t),\psi_2(t))\|^2_{\mathbb{D}}\\
=&\ \sigma_1(t)+\sigma_2(t)\\
&+\nu (\alpha+\beta) \int_{\mathbb{R}^N} h(x) \frac{(\psi_1^+(t))^\alpha (\psi_2^+(t))^\beta}{|x|^{s_3}} \, dx.
\end{split}
\end{equation}
Using H\"older's inequality, given in \eqref{GenHol}, one can bound the previous integral as
\begin{equation}\label{MPgeomp4}
\int_{\mathbb{R}^N} h(x) \frac{(\psi_1^+(t))^\alpha (\psi_2^+(t))^\beta}{|x|^s_3} \, dx \leq \nu \|h\|_{\mathfrak{p},\sigma}  (\sigma_1(t))^{\frac{\alpha}{2}\frac{N-2}{N-s_1}} (\sigma_2(t))^{\frac{\beta}{2}\frac{N-2}{N-s_2}}, \end{equation}
with $\mathfrak{p},\sigma$ defined \eqref{sigmadef}. 

Note that, from the definition of $\psi$, we have
\begin{equation*}
\sigma(0)=\left(\int_{\mathbb{R}^N} \frac{(z_1^{(1)})^{2^*_{s_1}}}{|x|^{s_1}} \, dx,0\right) \quad \mbox{ and } \quad \sigma(1)=\left(0,\int_{\mathbb{R}^N} \frac{(z_1^{(2)})^{2^*_{s_2}}}{|x|^{s_2}} \, dx \right).
\end{equation*}
As $\sigma(t)$ is continuous, there exists $\tilde{t}\in(0,1)$ such that 
$$
\frac{\sigma_1(\tilde{t})}{p_1}=\tilde{\sigma}=\frac{\sigma_2(\tilde{t})}{p_2} \qquad \mbox{where} \qquad p_j=\mathcal{S}(\lambda_j,s_j)^{\frac{N-s_j}{2-s_j}}.
$$ 

Taking $t=\tilde{t}$ in inequality \eqref{MPgeomp3} and applying \eqref{MPgeomp4}, we have that
\begin{equation*}
 p_1 \tilde \sigma^{\frac{2}{2^*_{s_1}}} +p_2 \tilde \sigma^{\frac{2}{2^*_{s_2}}} \leq (p_1+p_2) \tilde{\sigma} + C \nu (\alpha+\beta) \tilde{\sigma}^{\frac{\alpha}{2^*_{s_1}}+\frac{\beta}{2^*_{s_2}}}.
\end{equation*}
Since $\tilde{\sigma}\neq 0$, by Lemma~\ref{algelemma}, for some $\tilde{\nu}>0$ sufficiently small the previous inequality implies 

\begin{equation*}\label{MPgeomp5}
\tilde{\sigma}> 1-\varepsilon  \qquad \mbox{ for every } 0<\nu\le\tilde{\nu},
\end{equation*}
which implies
\begin{equation}\label{MPgeomp6}
\sigma_j(\tilde t)> (1-\varepsilon) \left[ \mathcal{S}(\lambda_j,s_j) \right]^{\frac{N-s_j}{2-s_j}}  \qquad \mbox{ for every } 0<\nu\le\tilde{\nu}.
\end{equation}
 As a result, from \eqref{MPgeomp2} and \eqref{MPgeomp6}, we deduce
\begin{equation*}
\begin{split}
\max_{t\in[0,1]} \mathcal{J}^+_\nu(\psi(t)) >&\ (1-\varepsilon) \left( \bigg(\frac{1}{2}-\frac{1}{2^*_{s_1}}\bigg) \left[ \mathcal{S}(\lambda_1,s_1) \right]^{\frac{N-s_1}{2-s_1}} +  \bigg(\frac{1}{2}-\frac{1}{2^*_{s_2}}\bigg) \left[ \mathcal{S}(\lambda_2,s_2) \right]^{\frac{N-s_2}{2-s_2}} \right)\\
>&\ 2 \mathfrak{C}(\lambda_2,s_2)\\
>&\ \mathfrak{C}(\lambda_1,s_1).
\end{split}
\end{equation*}
Then, $c_{MP}> \mathfrak{C}(\lambda_1,s)=\max\{\mathcal{J}^+_\nu(z_1^{(1)},0),\mathcal{J}^+_\nu(0,z_1^{(2)})\}$. Thus, $\mathcal{J}^+_\nu$ admits a Mountain--Pass structure on $\mathcal{N}_\nu$. 

\textbf{Step 2:} We consider the path $ \psi(t) =( \psi_1(t),   \psi_2(t))=\left((1-t)^{1/2} z_1^{(1)},t^{1/2}z_1^{(2)} \right)$ for $t\in[0,1]$. Using the property \eqref{normH}, we infer the existence of a positive function $\gamma:[0,1]\mapsto(0,+\infty)$ such that $\gamma \psi (t) \in \mathcal{N}_\nu^+\cap \mathcal{N}_\nu$ for every $t\in[0,1]$. Let us stress that $\gamma(0)=\gamma(1)=1$. As we did above, define
\begin{equation*}
\sigma(t)=(\sigma_1(t),\sigma_2(t))=\left(\int_{\mathbb{R}^N}\frac{\left( \gamma \psi_1(t)\right)^{2^*_{s_1}}}{|x|^{s_1}} \, dx , \int_{\mathbb{R}^N} \frac{\left( \gamma \psi_2(t)\right)^{2^*_{s_2}}}{|x|^{s_2}} \, dx \right).
\end{equation*}
By \eqref{normcrit}, we have that
\begin{equation}\label{Mpgeom7}
\sigma_1(0)=[\mathcal{S}(\lambda_1,s_1)]^{\frac{N-s_1}{2-s_1}}\quad\text{and}\quad \sigma_2(1)=[\mathcal{S}(\lambda_2,s_2)]^{\frac{N-s_2}{2-s_2}}.
\end{equation}
Due to the expression \eqref{normH}, one obtains
\begin{equation*}
\begin{split}
\left\|\left((1-t)^{1/2} z_1^{(1)},t^{1/2}z_1^{(2)} \right)\right\|^2_\mathbb{D} &= (1-t)\sigma_1(0)+t\sigma_2(1) \\
&= (1-t)^{2^*_{s_1}/2} \gamma^{2^*_{s_1}-2}(t)\sigma_1(0) + t^{2^*_{s_2}/2}\gamma^{2^*_{s_2}-2}(t)\sigma_2(1)  \\
&\mkern+20mu + \nu (\alpha+\beta)\gamma^{\alpha+\beta-2} (t)(1-t)^{\alpha/2}t^{\beta/2} \int_{\mathbb{R}^N} h(x) \frac{ (z_1^{(1)})^\alpha (z_1^{(2)})^{\beta}}{|x|^{s_3}} dx,
\end{split}
\end{equation*}
implying that, for every  $t\in(0,1)$, it holds
\begin{equation}\label{gammabound}
(1-t)\sigma_1(0)+t\sigma_2(1)>(1-t)^{2^*_{s_1}/2} \gamma^{2^*_{s_1}-2}(t)\sigma_1(0) + t^{2^*_{s_2}/2}\gamma^{2^*_{s_2}-2}(t)\sigma_2(1).
\end{equation}

Next, we apply the energy formula \eqref{Nnueq}
\begin{equation}\label{Jgammapsibound}
\begin{split}
\mathcal{J}_\nu^+ (\gamma \psi(t))=&\ \left( \frac{1}{2}-\frac{1}{\alpha+\beta} \right)\|\gamma\psi(t) \|^2_\mathbb{D}\\
&+ \left(\frac{1}{\alpha+\beta}- \frac{1}{2^*_{s_1}} \right) \gamma^{2^*_{s_1}}(t)  \|\psi_1(t)\|_{2_{s_1}^*,s_1}^{2_{s_1}^*}+    \left(\frac{1}{\alpha+\beta}- \frac{1}{2^*_{s_2}}\right) \gamma^{2^*_{s_2}}(t) \|\psi_2(t)\|_{2_{s_2}^*,s_2}^{2_{s_2}^*}\vspace{0.3cm} \\
 =&\  \left( \frac{1}{2}-\frac{1}{\alpha+\beta} \right) \gamma ^2(t)  \left[(1-t) \sigma_1(0) + t \sigma_2(1) \right] \\
 &+  \left(\frac{1}{\alpha+\beta}- \frac{1}{2^*_{s_1}} \right) \gamma^{2^*_{s_1}}(t) (1-t)^{2^*_{s_1}/2} \sigma_1(0)   +  \left(\frac{1}{\alpha+\beta}- \frac{1}{2^*_{s_2}}\right) \gamma^{2^*_{s_2}}(t) t^{2^*_{s_2}/2} \sigma_2(1) \\
 =&  \left( \frac{1}{2}  \gamma^2(t)  (1-t) - \frac{1}{2^*_{s_1}}  \gamma^{2^*_{s_1}}(t) (1-t)^{\frac{2^*_{s_1}}{2}} \right) \sigma_1(0)  + \left( \frac{1}{2}  \gamma^2(t)  t - \frac{1}{2^*_{s_2}}  \gamma^{2^*_{s_2}}(t) t^{\frac{2^*_{s_2}}{2}} \right) \sigma_2(1) \\
 & -\frac{1}{\alpha+\beta} \left( \gamma ^2(t)  \left[(1-t) \sigma_1(0) + t \sigma_2(1) \right] - \gamma^{2^*_{s_1}}(t) (1-t)^{\frac{2^*_{s_1}}{2}} \sigma_1(0) - \gamma^{2^*_{s_2}}(t) t^{\frac{2^*_{s_2}}{2}} \sigma_2(1) \right) .
\end{split}
\end{equation}

By using inequality \eqref{gammabound}, the last term of the above expression becomes negative. Then we can bound the energy as 
\begin{equation}\label{claim}
\begin{split}
  \mathcal{J}_\nu^+ (\gamma \psi(t))& < \left( \frac{1}{2}  \gamma^2(t)  (1-t) - \frac{1}{2^*_{s_1}}  \gamma^{2^*_{s_1}}(t) (1-t)^{\frac{2^*_{s_1}}{2}} \right) \sigma_1(0)  + \left( \frac{1}{2}  \gamma^2(t)  t - \frac{1}{2^*_{s_2}}  \gamma^{2^*_{s_2}}(t) t^{\frac{2^*_{s_2}}{2}} \right) \sigma_2(1)\\ & =g(t)
\end{split}
\end{equation}
for every $t\in(0,1)$.

Now, consider the function $h(x)=\frac{1}{2}x^2-\frac{1}{2^*_s}x^{2^*_s}$ with $s\in[0,2)$. Its maximum is attained at $1$, then $h(x)\le h(1)=\frac{1}{2}-\frac{1}{2^*_s}=\frac{2-s}{2(N-s)}$.

Since we can write
$$
g(t)=h(\gamma(t)(1-t)^{1/2})  \sigma_1(0) + h(\gamma(t)t^{1/2})\sigma_2(1),
$$
one infers that
$$
\mathcal{J}_\nu^+ (\gamma \psi(t))<\bigg(\frac{1}{2}-\frac{1}{2^*_{s_1}}\bigg) \sigma_1(0)+\bigg(\frac{1}{2}-\frac{1}{2^*_{s_2}}\bigg) \sigma_2(1) =  \mathfrak{C}(\lambda_1,s_1)+\mathfrak{C}(\lambda_2,s_2),
$$
using \eqref{Jgammapsibound} and  \eqref{lamdasalphabeta}.
Consequently,
$$
\mathfrak{C}(\lambda_2,s_2)<\mathfrak{C}(\lambda_1,s_1)<c_{MP} \leq \max_{t\in[0,1]} \mathcal{J}_\nu^+(\gamma \psi(t))<  3 \mathfrak{C}(\lambda_2,s_2).
$$
Then, the  Mountain--Pass level $c_{MP}$ satisfies the assumptions of Lemmas~\ref{lemmaPS1} and \ref{lemcritic}. By the Mountain--Pass Theorem, we can infer the existence of a sequence $\left\{ (u_n,v_n) \right\} \subset \mathcal{N}^+_\nu$ such that 
\begin{equation*}
\mathcal{J}^+(u_n,v_n ) \to c_\nu \quad\text{and}\quad \mathcal{J}^+|_{\mathcal{N}
^+_\nu}(u_n,v_n ) \to 0.
\end{equation*}
If $\frac{\alpha}{2^*_{s_1}}+\frac{\beta}{2^*_{s_2}}<1$, by analogous versions of Lemmas~\ref{lemma:PSNehari} and \ref{lemmaPS1} for $\mathcal{J}^+_\nu$, we get $\left\{ (u_n,v_n) \right\} \to ( \tilde{u},\tilde{v} )$. Indeed, $(\tilde{u},\tilde{v} )$ is a critical point of $\mathcal{J}_\nu$ on $\mathcal{N}_\nu$ so it is also a critical point of $\mathcal{J}_\nu$ defined in $\mathbb{D}$. Moreover, $\tilde{u},\tilde{v} \ge 0$ in $\mathbb{R}^N$ and by the maximum principle in $\mathbb{R}^N\setminus \{0\}$ we conclude they are strictly positive.
For assumptions $ii)$, the PS condition follows by Lemma~\ref{lemmaPS1}. If $\frac{\alpha}{2^*_{s_1}}+\frac{\beta}{2^*_{s_2}}=1$, we follow the same approach using  now Lemma~\ref{lemcritic}.
\end{proof}

\noindent\textbf{Acknowledgements.}
\'A.~A.\ is supported by the  MICIN/AEI through the Grant PID2021-123151NB-I00. Part of this research was done during a visit of  \'A.~A.\ to the University of Granada in 2023. R.~L.-S. is  supported by the  MICIN/AEI through the Grant PID2021-122122NB-I00 and the \emph{IMAG-Maria de Maeztu} Excellence Grant CEX2020-001105-M, and by J. Andalucia via the Research Group FQM-116.  R.~L.-S. has also been supported by the grant Juan de la Cierva Incorporaci\'on fellowship (JC2020-046123-I), funded by MCIN/AEI/10.13039/501100011033, and by the European Union Next Generation EU/PRTR.

%\begin{center}{\bf Acknowledgements}\end{center}

\end{document}